\documentclass[a4paper, reqno, 11pt]{amsart}

\usepackage[english]{babel}
\usepackage{amsmath}
\usepackage{amssymb}
\usepackage{amsthm}
\usepackage{enumerate}
\usepackage{ifthen}
\usepackage{geometry}
\usepackage{bbm}
\usepackage{color}
\usepackage{hyperref}

\newcommand{\margnote}[1]{
\ifthenelse{\boolean{shownotes}}%
{\marginpar{\raggedright\tiny\texttt{#1}}}%
{}%
}

\theoremstyle{plain} \newtheorem{Theorem}{Theorem}
\theoremstyle{plain} 
\theoremstyle{plain} 
\theoremstyle{plain} 
\theoremstyle{definition} 
\theoremstyle{definition}
\theoremstyle{plain} 
\theoremstyle{plain} 
\theoremstyle{plain} \newtheorem{lemma}[Theorem]{Lemma}
\theoremstyle{plain} \newtheorem{proposition}[Theorem]{Proposition}
\theoremstyle{definition} \newtheorem{definition}[Theorem]{Definition}
\theoremstyle{definition}\newtheorem{remark}[Theorem]{Remark}
\theoremstyle{definition}

\newcommand{\curl}{\mathop{\mathrm {curl}}}
\newcommand{\nphi}{\nabla\varphi}
\newcommand{\Rdd}{{\mathbb{R}^{d\times d}}}
\newcommand{\Rd}{{\mathbb{R}^d}}
 
\newcommand{\RR}{{\mathbb{R}}}
\newcommand{\R}{{\mathbb{R}}}

\newcommand{\eps}{\varepsilon}
\newcommand{\ph}{\varphi}

\newcommand{\dv}{{\rm div \,}}

\newcommand{\GF}{G_{W}}
\newcommand{\tW}{\tilde{W}}
\newcommand{\tGF}{G_{\tW}}

\numberwithin{equation}{section}

\begin{document}

\title[Quasiconvex elastodynamics: weak-strong uniqueness]{Quasiconvex elastodynamics: weak-strong uniqueness for measure-valued solutions}

\author[K. Koumatos]{Konstantinos Koumatos}
\address[Konstantinos Koumatos]{\newline
Department of Mathematics, University of Sussex\\ Pevensey 2 Building,
Falmer,
Brighton, BN1 9QH, UK}
\email[]{\href{k.koumatos@sussex.ac.uk}{k.koumatos@sussex.ac.uk}}

\author[S. Spirito]{Stefano Spirito}
\address[Stefano Spirito]{\newline
Department of Information Engineering, Computer Science and Mathematics, University of L'Aquila,\\
Via Vetoio, Coppito, 67100, L'Aquila, Italy}
\email[]{\href{stefano.spirito@univaq.it}{stefano.spirito@univaq.it}}

\begin{abstract}
A weak-strong uniqueness result is proved for measure-valued solutions 
to the system of conservation laws arising in elastodynamics. 
The main novelty brought forward by the present work is that the underlying 
stored-energy function of the material is assumed strongly quasiconvex. 
The proof employs tools from the calculus of variations to establish general convexity-type bounds on quasiconvex functions 
and recasts them in order to adapt the relative entropy method to quasiconvex elastodynamics.
 \vspace{0.1cm}

 \vspace{0.1cm}

\end{abstract}

\maketitle

\section{Introduction}
For $d = 2,3$ let $Q=(0,1)^d$ and $Q_T = (0,T)\times Q$ for some arbitrary finite $T>0$.  
For $(t,x)\in Q_T$ and $S:\Rdd\to\Rdd$ a given mapping,
we consider the system of conservation laws
\begin{equation}\label{eq:intro_cl}
\begin{aligned}
\partial_t u(t,x)-\dv S(F(t,x))&=0,\\
\partial_t F(t,x)-\nabla u(t,x)&=0,\\
u(0,x) &= u^0(x),\\
F(0,x) &= F^0(x),
\end{aligned}
\end{equation}
for the unknown functions $u:Q_T\to\Rd$ and $F:Q_T\to\Rdd$ under periodic boundary conditions. Imposing the additional constraint
\begin{equation}\label{eq:inv_intro}
\curl F(t,x)=0\textrm{ for any }t\in(0,T),
\end{equation}
system \eqref{eq:intro_cl} reduces to the equations of motion of a (homogeneous) hyperelastic body in the absence of external forces.
In this context, the mapping $S$ expresses the Piola-Kirchhoff stress tensor which, under the assumption of hyperelasticity, is given by
\begin{equation*}
S(F) = D W(F), \quad F\in \Rdd,
\end{equation*}
where $W:\Rdd\to\RR$ models the stored-energy function of the material.
Indeed, by setting  $F=\nabla y$ and $u=\partial_t y$, for some function $y:\Rd\to\Rd$ representing the deformation of the body, it follows that $y$ satisfies the quasi-linear wave equation
\begin{equation}\label{eq:intro_wave}
\frac{\partial^2y(t,x)}{\partial t^2} - \dv S(\nabla y(t,x)) = 0
\end{equation}
which is the standard form of Cauchy's equations of motion in elasticity.
It is important to point out that the constraint \eqref{eq:inv_intro} is an involution of system \eqref{eq:intro_cl}, meaning that if the initial data $F^0$ are curl-free, the evolution preserves the constraint for the solution $F$, see e.g. \cite{dafermos_involutions}.
The aim of this paper is to study the question of weak-strong uniqueness for measure-valued solutions to system \eqref{eq:intro_cl} in $(0,T)\times Q$
under the assumption of (strong) quasiconvexity for the stored-energy function $W$. 
The notion of measure-valued solutions was originally introduced by DiPerna in \cite{DiPerna85} for conservation laws and by DiPerna \& Majda in \cite{DiPerna_Majda} for the Euler equations and, 
although it is a weak notion of solution, it is one that allows for a global existence theory in many physical systems.
The question of weak-strong uniqueness is then natural as the minimal requirement for any notion of solution, namely that it must agree with the classical solution whenever the latter exists and has gained much attention in recent years, see \cite{BrenierDeLellis_CMP}, \cite{TzavarasARMA2012}.

In the theory of conservation laws as well as in the equations of fluid dynamics, convexity of the energy is key to the analysis. In particular, the natural bounds that convexity induces on the energy allow for stability and weak-strong uniqueness results to be established via an application of the so-called relative entropy method, see \cite{dafermos_book}, a tool that has also proved useful in treating singular limits \cite{lattanzio_tzavaras2006}.
However, in nonlinear elasticity, the energy associated to system \eqref{eq:intro_cl} takes the form
\[
\frac{1}{2}|u|^2 + W(F)
\]
and convexity of the stored-energy function $W$ is seen as inconsistent with frame-indifference since it imposes stringent positivity conditions on the stress, see e.g. \cite[Proposition 17.5.3]{silhavy_book}.

Instead, a natural notion of convexity in elasticity is that of quasiconvexity (Definition \ref{def:QC}) - a condition strictly weaker than convexity for $d\geq 2$. Indeed, at least in the static case, \eqref{eq:intro_wave} reduces to the system
\begin{equation}\label{eq:intro_static}
-\dv S(\nabla y) = 0
\end{equation}
with the associated variational problem of minimizing the energy functional
\[
\mathcal{E}(y):=\int W(\nabla y).
\]
In this context and modulo growth conditions, the assumption of quasiconvexity on $W$ is equivalent to the weak lower semicontinuity of $\mathcal E$ in the Sobolev space $W^{1,p}$, $p\in(1,\infty)$, which provides the existence of minimizers via the direct method. In fact, quasiconvexity is \emph{almost} necessary for the existence of minimizers as \cite[Corollary 5.2]{Ball_Murat} suggests. It is hence not only natural to consider the problem of quasiconvex elastodynamics but to also conjecture that the quasiconvexity of $W$ must endow the dynamics with better properties; the weak-strong uniqueness result proved in the present article being one such property. As a matter of fact and to the best of the authors' knowledge, this is the first result in which quasiconvexity is explicitly used in the evolution problem.

In fact, the non-local nature of quasiconvexity \cite{kristensen_nonlocal} poses great difficulties and other convexity conditions have been introduced, namely polyconvexity and rank-one convexity, satisfying the following chain of implications: 
\begin{equation*}
\textrm{convexity}\Rightarrow\textrm{polyconvexity}\Rightarrow\textrm{quasiconvexity}\Rightarrow\textrm{rank-one convexity}.
\end{equation*}
We note that all reverse implications are known to be false, apart for the case of rank-one convexity implying quasiconvexity and $d=2$ which remains an open problem, see \cite{sverak_counterexample}. For precise definitions as well as proofs of the above implications and counterexamples, we refer the reader to \cite{mullernotes,sverak_counterexample}.

In terms of the evolution problem \eqref{eq:intro_cl} and the above convexity conditions, local existence of classical solutions for the Cauchy problem has been shown in  \cite{dafermos_book} - see also \cite{HuKaMa1977} for the wave equation \eqref{eq:intro_wave} - under the assumption of (strong) rank-one convexity on $W$ (in particular, quasiconvexity)\footnote{We also refer the reader to the monograph of Valent \cite{valent_book} where local well-posedness results are proved for the boundary value problem through the Implicit Function theorem and without any constitutive assumptions on $W$ in the form of convexity conditions.}. In this case, however, weak-strong uniqueness can only be established assuming small enough shocks in the weak solution, see \cite{dafermos_book,dafermos_involutions}, although the global existence of weak solutions is an open problem even in the case of convex $W$. 
In drawing analogues between statics and dynamics, we remark that the rank-one convexity of $W$ makes the static problem \eqref{eq:intro_static} elliptic and the evolution problem \eqref{eq:intro_cl} hyperbolic.

Regarding polyconvexity, many physical energies fall into this category and the static theory admits minimizers even under the physical assumption that the energy density $W$ blows up as $\det\nabla y\to0^+$, see Ball \cite{Ball_existence}. The problem of extending Ball's seminal result to quasiconvex functions, remains an important open problem in elastostatics and we will not be concerned with it here. For polyconvex energies and the evolution problem, the existence and weak-strong uniqueness of measure-valued solutions for the initial boundary value problem on the flat torus has been shown by Demoulini, Stuart \& Tzavaras in \cite{TzavarasARMA2001} and \cite{TzavarasARMA2012}, respectively. In particular, the weak-strong uniqueness result in \cite{TzavarasARMA2012} employs the relative entropy method and the convexity of the energy for an enlarged system whose involutions make it equivalent to \eqref{eq:intro_cl}. 

As a further motivation for the use of measure-valued solutions as well as our result, we note that the variational principle in elastostatics is motivated by an (in general only formal) argument showing that the dynamics produce infimizing sequences for the energy $\mathcal E$ so that, in the limit $t\to\infty$, minimizers of the energy are attained when these exist. The rigorous justification of the variational principle is an open problem in elasticity and the reader is referred to \cite{Ball_open_problems} and references therein for a discussion. However, as mentioned above, in the absence of quasiconvexity, the functional $\mathcal E$ may not admit minimizers. Instead, the gradients of infimizing sequences typically develop oscillations and then it is the generated Young measures that minimize the relaxed problem
\[
\mathcal E^{\rm rel} = \int \langle\nu_x, W\rangle.
\] 
This is precisely the framework under which microstructure in materials undergoing martensitic transformations is modelled, see e.g. \cite{BallJames92}. In this context, it is thus not unreasonable to consider measure-valued solutions in elastodynamics. Nevertheless, if $W$ is quasiconvex and minimizers do exist then it is also natural to expect that the dynamics should produce stronger solutions than measures and that the measure-valued solutions should collapse to this stronger solution. This may serve as an interpretation of the weak-strong uniqueness result, although it is unknown whether such strong solutions exist globally under the quasiconvexity assumption.

In the present work, we consider the flat torus as our spatial domain and a suitable notion of dissipative measure-valued solutions to system \eqref{eq:intro_cl} is defined (see Definition \ref{def:mvs}). For these solutions, a weak-strong uniqueness result is established for stored-energy functions $W$ which are strongly quasiconvex. The defined measure-valued solutions assume two additional properties compared to standard definitions, see e.g. \cite{TzavarasARMA2012}, which are natural in the sense that any reasonable approximating system will fulfil these requirements. On the one hand, we assume that the measure-valued solutions are generated by a sequence of spatial gradients.
Due to the induced involution of the system, this is natural but also essential in order to use the quasiconvexity assumption.
On the other hand, the generating sequences are also required to enjoy a certain time regularity, in particular 
$(\partial_t F^n)_{n}\subset L^{\infty}(0,T; H^{-1}(Q))$. This condition should also be satisfied by reasonable approximations and it establishes an equivalence between measure-valued solutions of the wave equation \eqref{eq:intro_wave} and the system of conservation laws \eqref{eq:intro_cl}.
We remark that existence of such dissipative measure-valued solutions is simple to obtain and it is thus not addressed in the current work (see Remark \ref{rem:existence}). 

The article is organized as follows: in Section \ref{sec:prelim}, we introduce some preliminary notation, definitions and tools. These include our functional setting as well as a brief summary on Young measures and quasiconvexity. In Section \ref{sec:solutions}, we lay out all assumptions made on the stored-energy function $W$ and we define the notion of measure-valued solutions for system \eqref{eq:intro_cl}. In Section \ref{sec:main}, we state and prove the weak-strong uniqueness theorem for the defined dissipative measure-valued solutions which is the main result of the paper. The proof of this result is based on a variant of the relative entropy method, however, the lack of convexity of $W$ presents a crucial obstacle. This obstacle is overcome by noting that the pointwise bounds provided by convexity are not required but an averaged version of them suffices. This is precisely the content of Theorem \ref{theorem:2} where these averaged convexity-type bounds on the quasiconvex stored-energy $W$ are established. We stress that Theorem \ref{theorem:2} is independent of the equations and it is of broader interest. Its proof, motivated by the works in \cite{JCCVMO,JCCKK,AcFus87,Zhang92}, is based on the calculus of variations and it is postponed until Section \ref{sec:propproof} where a precise statement is also given. 


\section{Notations and Pleliminaries}\label{sec:prelim}
In this section we fix the notation used in the paper and we recall definitions and useful facts about quasiconvex functions and Young measures which are used in the sequel. 
\subsection{Notation and function spaces}
We denote by $C^{k}(Q)$ and $C^{\infty}(Q)$ the spaces of $k$-times continuously differentiable and smooth functions, respectively, which are $Q$-periodic. 
We denote by $L^{p}(Q)$ the standard Lebesgue spaces of $Q$-periodic functions and by $\|\cdot\|_{L^{p}(Q)}$ their norm. The Sobolev space of $L^{p}$ $Q$-periodic functions with $k$ distributional derivatives in $L^{p}$ is denoted by $W^{k,p}(Q)$ and their norms by $\|\cdot\|_{k,p}$. In the case $p=2$ and $k=1$ we denote by $H^{1}_{0}(Q)$ the space of periodic functions in $W^{1,2}(Q)$ with zero average and the spatial average on the torus $Q$ of a function $f$ is denoted by $(f)_Q$. Finally, we let $H^{-1}(Q):=(H^{1}_{0}(Q))'$. Concerning the time dependence, we consider the classical Bochner spaces $L^p(0,T;X)$, endowed with
the norm
\begin{equation*}
  \|f\|_{L^p(X)}:=
  \left\{\begin{aligned}
      &\left(\int_0^T\|f(s)\|^p_X\right)^{1/p}\quad &\text{if }1\leq p<\infty,
      \\
      &\sup_{0\leq s\leq T}\|f(s)\|_X &\text{if }p=+\infty,
\end{aligned}
\right.
\end{equation*}
where $X$ is a Banach space. In particular, when $X=L^{p}(Q)$ the norm of $L^{p}(0,T;L^{p}(Q))$ is denoted by $\|\cdot\|_{L^{p}(Q_T)}$. Whenever the target space is clear from the context we will not distinguish between scalar, vector and matrix-valued spaces. Finally, for a general regular domain $\Omega\subset\RR^d$ which is not the $d$-dimensional flat torus, the space $W^{1,p}_{0}(\Omega)$ denotes the space of Sobolev functions in $W^{1,p}(\Omega)$ whose trace on the boundary of $\Omega$ vanishes. 
\subsection{Quasiconvexity}

Throughout we assume that $p\geq 2$ and we define the auxiliary function $V:\RR^k\to\RR$ by
\begin{equation}\label{eq:defV}
V(\xi) = (|\xi|^2 + |\xi|^p)^{1/2}
\end{equation}
where $k = d$ or $k = d\times d$.

\begin{definition}\label{def:QC}
A continuous function $W:\Rdd\to\RR$ is quasiconvex at the matrix $\xi\in\Rdd$ if the inequality
\[
\int_Q \big[W(\xi + \nabla\ph(x)) - W(\xi)\big] dx \geq 0
\]
holds for all $\ph\in W^{1,\infty}(Q)$. The function $W$ is called quasiconvex if it is quasiconvex at each $\xi\in\Rdd$. If, in addition, there exists a constant $c_0>0$ such that
\[
\int_Q \big[W(\xi + \nabla\ph(x)) - W(\xi) \big] dx \geq c_0 \int_Q |V(\nabla\ph(x))|^2 dx
\]
holds for all $\xi\in\Rdd$ and all $\ph\in W^{1,\infty}(Q)$, we say that $W$ is strongly quasiconvex.
\end{definition}

\begin{remark}
Quasiconvexity is usually defined through test functions $\ph\in W_{0}^{1,\infty}(\Omega)$, with 
$\Omega\subset\RR^d$ a bounded domain. We remark that the definition presented above in terms of $Q$-periodic functions is equivalent and we refer the reader to \cite[Proposition 5.13]{dacorogna} for a proof.
\end{remark}

With the above definition at hand, we next present a lemma listing some crucial properties of $W$ under quasiconvexity and growth assumptions. All properties are standard and we refer the reader to \cite{dacorogna} for the proofs in the case of quasiconvex functions; the extension to strongly quasiconvex functions is analogous and the last assertion is a corollary of (3).

\begin{lemma}\label{lemma:prelim}
Suppose that $W:\Rdd\to\RR$ is continuous, strongly quasiconvex and satisfies
\[
|W(\xi)| \leq c(1 + |\xi|^p).
\]
Then the following hold:
\begin{enumerate}
\item the defining inequality
\[
\int_Q \big[W(\xi + \nabla\ph(x)) - W(\xi) \big] dx \geq c_0 \int_Q |V(\nabla\ph(x))|^2 dx
\]
holds for all $\ph\in W^{1,p}(Q)$, i.e. $W$ is $p$-quasiconvex;
\item for any $x_0\in\RR^d$ and $r>0$, denoting $Q(x_0,r) = x_0 + rQ$, it holds that
\[
\int_{Q(x_0,r)} \big[W(\xi + \nabla\ph(x)) - W(\xi) \big] dx \geq c_0 \int_{Q(x_0,r)} |V(\nabla\ph(x))|^2 dx,
\] 
for all $\varphi \in W_{0}^{1,p}(Q(x_0,r))$.
\item for some constant $c>0$ and every $\xi,\eta\in\Rdd$
\[
|W(\xi) - W(\eta)|\leq c (1 + |\xi|^{p-1} + |\eta|^{p-1})|\xi - \eta|;
\]
\item if, in addition, $W\in C^1(\Rdd,\RR)$ then
\[
|DW(\xi)|\leq c(1+|\xi|^{p-1}).
\]
\end{enumerate}
\end{lemma}

\subsection{Young Measures}

For $q\geq 0$ and $m,k\geq 1$ arbitrary, we let $C_q(\RR^k)$ denote the subspace of continuous functions on $\RR^k$, $C(\RR^k)$, given by
\begin{equation}\label{eq:Cq}
C_q(\RR^k):=\left\{g\in C(\RR^k)\,:\,\lim_{|\xi|\to\infty}\frac{g(\xi)}{|\xi|^q} = 0\right\}.
\end{equation}
Under the above notation, the space $C_0(\RR^k)$ denotes the space of continuous functions `vanishing at infinity' and it can be identified with the completion of compactly supported, continuous functions in the $L^\infty$-norm. By the Riesz representation theorem, its dual, $C_0(\RR^k)^*$ is isometrically isomorphic to the space of signed Radon measures on $\RR^k$, $\mathcal M(\RR^k)$, equipped with the total variation norm.

Let $\Omega\subset \RR^m$ be a bounded domain and denote by 
\[
L^\infty_{w*}(\Omega,\mathcal M(\RR^k))
\] 
the space of essentially bounded, weakly-$*$ measurable maps from $\Omega$ into $\mathcal M(\RR^k)$, i.e. those mappings $\nu: x\mapsto \nu_x\in \mathcal M(\RR^k)$ such that
\begin{itemize}
\item ${\displaystyle \sup_{x\in\Omega} \|\nu_x\|_{\mathcal M(\RR^k)} < \infty}$;
\item for all $g\in C_0(\RR^k)$ the function
\[
x\mapsto \langle\nu_x, g\rangle:=\int_{\RR^k} g(\xi) d\nu_x(\xi)
\]
is measurable.
\end{itemize}

A Young measure $\nu = (\nu_x)_{x\in\Omega}$ on $\Omega$ is an element of $L^\infty_{w*}(\Omega,\mathcal M(\RR^k))$ taking values in the space of probability measures, i.e.
\[
\|\nu_x\|_{\mathcal M(\RR^k)} = 1\mbox{ a.e. in $\Omega$.}
\]
Note that by the separability of $C_0(\RR^k)$, it holds that
\[
L^\infty_{w*}(\Omega,\mathcal M(\RR^k)) = L^1(\Omega,C_0(\RR^k))^*
\]
and this duality defines a weak-$*$ convergence of Young measures. Then, the fundamental theorem of Young measures, see e.g. \cite{BallYM}, states that given a sequence $(Y_n)$ bounded in $L^q(\Omega,\RR^k)$, $1\leq q < \infty$, there exists a subsequence and a Young measure $\nu = (\nu_x)_{x\in\Omega}$ such that
\begin{equation}\label{eq:generationYM}
g(Y_n) \rightharpoonup \langle\nu_x,g\rangle \mbox{ in }L^1(\Omega)\mbox{ for all }g\in C_q(\RR^k).
\end{equation}
We note that $C_q(\RR^k)$ is itself separable when equipped with the norm $\|g(\cdot)/(1+|\cdot|^q)\|_{\infty}$ and that
the above convergence also holds whenever the sequence $(g(Y_n))$ is equiintegrable.
In particular, the barycentre $\langle\nu_x,{\rm id}\rangle$ of the generated Young measure identifies the weak limit of the sequence $(Y_n)$, i.e. 
\[
Y_n \rightharpoonup \langle\nu_x,{\rm id}\rangle \mbox{ in }L^q(\Omega).
\]
Whenever \eqref{eq:generationYM} holds, we say that the sequence $(Y_n)$ generates the Young measure $\nu$ and we call $\nu$ a $q$-Young measure, i.e. a Young measure generated by a sequence bounded in $L^q(\Omega)$. If, in addition, $Y_n = \nabla y_n$ for some $y_n\in W^{1,q}(\Omega)$ then we call $\nu$ a gradient $q$-Young measure. It can also be shown, see e.g. \cite{kristensen99,kruzik&rubicek}, that every Young measure $\nu$ satisfying
\[
\int_\Omega \langle\nu_x,|\cdot|^q\rangle \,dx <\infty
\]
is indeed a $q$-Young measure. We note that similar statements hold for $q=\infty$ but we will only be concerned with finite exponents. 

In the sequel, we are interested in Young measures generated by sequences $(Y_n)$ bounded in the Bochner space $L^\infty(0,T;L^q(Q))$. Of course, $(Y_n)$ is then also bounded in $L^q(Q_T)$ and generates a $q$-Young measure $\nu=(\nu_{t,x})_{(t,x)\in Q_T}$ satisfying
\[
\int_{Q_T} \langle\nu_{t,x},|\cdot|^q\rangle \,dx <\infty.
\]
The following lemma, which is crucial in our analysis, shows that the above integrability can be improved to obtain $L^\infty$ bounds in the time variable. Its proof can be found in \cite{BrenierDeLellis_CMP}.

\begin{lemma}\label{lemma:supboundint}
Let $(Y_n)$ be a sequence of functions bounded in $L^\infty(0,T;L^q(Q))$, generating the $q$-Young measure $\nu=(\nu_{t,x})_{(t,x)\in Q_T}$ in $L^q(Q_T)$. Then
\[
 \sup_t \int_Q \langle\nu_{t,x},|\cdot|^q\rangle\,dx < \infty. 
\]
\end{lemma}

We end this section by remarking that, in our context, 
\[
Y_n = (u^n,F^n)\in L^\infty(0,T;L^2(Q)\times L^p(Q)),
\]
where $u^n\in\Rd$ and $F^n\in\Rdd$. Then, the sequence $(Y_n)$ generates a Young measure $\nu$ and the convergence
\[
g(Y_n) \rightharpoonup \langle\nu_x,g\rangle \mbox{ in }L^1(Q_T)
\]
holds for all $g\in C(\Rd\times\Rdd)$ such that
\[
\lim_{|\lambda| + |\xi|\to\infty} \frac{g(\lambda,\xi)}{|\lambda|^2+|\xi|^p}=0.
\]
Also, once $(Y_n)$ is bounded in $L^\infty(0,T;L^2(Q)\times L^p(Q))$, it holds that
\[
Y_n \rightharpoonup Y = (u,F)\in L^\infty(0,T;L^2(Q)\times L^p(Q)),
\]  
where, denoting by $\pi_d:\Rd\times\Rdd\rightarrow\Rd$ and $\pi_{d\times d}:\Rd\times\Rdd\rightarrow\Rdd$ the projections onto $\Rd$ and $\Rdd$, respectively, 
\[
u(t,x) = \langle\nu_{t,x},\pi_d\rangle\quad\mbox{and}\quad F(t,x) = \langle\nu_{t,x},\pi_{d\times d}\rangle\mbox{ a.e. in $Q_T$}.
\]

\section{Definition of Measure-Valued Solutions}\label{sec:solutions}
In this section we give the precise definition of measure-valued solutions for the system under consideration. 
We start by writing system \eqref{eq:intro_cl} in a precise form:
in $(0,T)\times Q$ we consider the following initial value problem for $Q$-periodic functions:
\begin{equation}\label{eq:cl}
\begin{aligned}
\partial_t u_i-\partial_{\alpha} S_{i\alpha}(F)&=0,\\
\partial_t F_{i\alpha}-\partial_{\alpha} u_i&=0,\\
u|_{t=0}&=u^0,\\
F|_{t=0}&=F^0. 
\end{aligned}
\end{equation}
Moreover, we assume that $F$ satisfies the involution \eqref{eq:inv_intro}, namely we assume that there exists $y:\Rd\to\Rd$ such that 
\begin{equation}\label{eq:inv}
F(t,x)=\nabla y(t,x)\textrm{ for }(t,x)\in(0,T)\times Q
\end{equation}
and, without loss of generality, we may assume that 
\begin{equation}\label{eq:zeroav}
\int_{Q}\,y(t,x)\,dx=0.
\end{equation}
Concerning the tensor $S$ we assume that for $F\in\Rdd$
\begin{equation*}
S(F)=DW(F),\text{ or equivalently, } S_{i\alpha}(F) = \frac{\partial W(F)}{\partial F_{i\alpha}}
\end{equation*}
and that the stored-energy function $W:\Rdd\to\RR$ satisfies the following:
\begin{itemize}
\item[(H1)] $W\in C^3(\Rdd)$;
\item[(H2)] $W$ is strongly quasiconvex with constant $c_0>0$;
\item[(H3)] $ |W(\xi)| \leq c(1 + |\xi|^p)$ and $|D^2W(\xi)|\leq c(1 + |\xi|^{p-1})$;
\item[(H4)] $c(|\xi|^p-1) \leq W(\xi)$.
\end{itemize}
\begin{remark}\label{rem:secondgrowth}
We remark that the assumed growth on the second derivative of $W$ in (H3) is not redundant. Indeed, there exist strongly quasiconvex functions with $p$-growth, yet with no polynomial control on the second derivative, see \cite{AcFus87}.
\end{remark}
Concerning the initial data for the velocity, we assume that 
\begin{equation}\label{eq:idu}
u^0\in L^{2}(Q)\textrm{ and }\int_{Q}u^0=0,\\
\end{equation}
and for the deformation tensor we assume that there exists $y^0\in H^{1}_{0}(Q)\cap W^{1,p}(Q)$ such that
\begin{equation}\label{eq:idf}
F^0=\nabla y^0.
\end{equation}
We recall that system \eqref{eq:cl} is endowed with a natural entropy-entropy flux pair $(\eta, q)$ given by
\[
\eta(u,F) = \frac12 |u|^2 + W(F)\mbox{ and } q(u,F) = u^TS(F),
\] 
i.e. $q$ is a vector in $\Rd$ with components
\[
q_\alpha = u_i S(F)_{i\alpha}.
\]
In particular, any classical solution to \eqref{eq:cl} - that is a pair $(u,F)$ of Lipschitz functions on $\overline{Q}_T$, periodic on $Q$ satisfying \eqref{eq:cl} - will automatically satisfy
\begin{equation}\label{entropy_eqn}
\partial_t \eta(u,F) + \dv q(u,F) = 0.
\end{equation}
In order to define a measure-valued solution we remark that natural approximations of \eqref{eq:cl} should produce a sequence of functions $(u^{n}, F^{n})$ such that 
\begin{equation}\label{eq:enen}
\sup_n \sup_{t\in(0,T)}\int_Q\eta(u^n,F^n)\,dx<\infty
\end{equation}
and therefore the following uniform bounds hold: 
\begin{equation*}
\begin{aligned}
&u^n\in L^{\infty}(0,T;L^{2}(Q)),\\
&F^n\in L^{\infty}(0,T;L^{p}(Q)).
\end{aligned}
\end{equation*}
Another natural uniform estimate for the approximation is a bound on the time derivatives of $u^n$ and $F^n$ in some negative Sobolev space. For our purposes, it is enough to assume that the Young measure is generated by a sequence $(u^n, F^n)$ satisfying 
the uniform bound
\begin{equation}\label{eq:trn}
\partial_t F^n\in L^{\infty}(0,T;H^{-1}(Q)). 
\end{equation}
The definition of dissipative measure-valued solutions for the initial boundary value problem \eqref{eq:cl}-\eqref{eq:inv}-\eqref{eq:zeroav} follows: 
\begin{definition}\label{def:mvs}
The triple $(u,F,\nu)$ is a dissipative measure-valued solution of \eqref{eq:cl}-\eqref{eq:inv}-\eqref{eq:zeroav} with initial data $(u^0, F^0)$ satisfying \eqref{eq:idu} and \eqref{eq:idf} if the following properties hold:
\begin{enumerate}
\item {\em Integrability hypothesis}\\
The vector field $u$ lies in $L^{\infty}(0,T;L^{2}(Q))$ and the matrix field $F$ in $L^{\infty}(0,T;L^{p}(Q))$;
\item {\em Equations}\\
For every $\ph\in C^{\infty}_c([0,T);C^{\infty}(Q))$ and $\Phi\in C^{\infty}_c([0,T);C^{\infty}(Q))$, the triple $(u,F,\nu)$ satisfies 
\begin{align*}
&\int_Q u^0\cdot\ph(0,\cdot) dx + \int_0^T\int_Q u \cdot \partial_t \ph dx dt = \int_0^T\int_Q \langle \nu_{t,x}, S\rangle\cdot\nabla\ph dx dt\\
&\int_Q F^0\cdot\Phi(0,\cdot) dx + \int_0^T\int_Q F\cdot \partial_t \Phi dx dt = \int_0^T\int_Q u\cdot \dv\Phi dx dt,
\end{align*}
where $u = \langle\nu, \pi_d\rangle$ and $F = \langle\nu,\pi_{d\times d}\rangle$;

\item {\em Generation of Young measure}\\
The Young measure $\nu = (\nu_{t,x})_{(t,x)\in Q_T}$ is generated by a sequence $(u^n, F^n)$ satisfying \eqref{eq:enen} and \eqref{eq:trn}. 
\item {\em Energy dissipation}\\
There exists a nonnegative Radon measure $\gamma$ such that the inequality
\[
\int_Q \theta(0)\eta(u^0,F^0) dx + \int_0^T\int_Q \dot{\theta}\left\{ \langle\nu_{t,x},\eta\rangle \,dx dt + \gamma(dxdt)\right\} \geq 0
\]
holds for all nonnegative functions $\theta\in C^1_c([0,T))$. 
\end{enumerate}
\end{definition}

\begin{remark}\label{rem:existence}
To prove existence of such measure-valued solutions it suffices to approximate system \eqref{eq:intro_cl} by, for example, the 4th order regularisation
\begin{align*}
\frac{\partial u^\eps}{\partial t} - \dv S(F^\eps) & = \eps\Delta u^\eps - \eps\Delta(\Delta u^\eps) \\
\frac{\partial F^\eps}{\partial t} - \nabla u^\eps & = 0.
\end{align*}
We note that, under the above approximation, existence can be established assuming only the smoothness and $(p-1)$-growth on $S$ whereas the quasiconvexity of $W$
is not required.
\end{remark}

\section{Main result}\label{sec:main}
In this section we state and prove our main result concerning the weak-strong uniqueness for measure-valued solutions. Its precise statement follows:
\begin{Theorem}\label{teo:main}
Let $(\bar u^0,\bar F^0)\in W^{1,\infty}(Q)$ such that $\bar F^0=\nabla \bar y^0$ for some zero-average vector $\bar y^0\in W^{2,\infty}(Q)$ and
let $(\bar{u},\bar{F})\in W^{1,\infty}([0,T]\times \overline{Q})$ be a classical solution to \eqref{eq:cl} with initial data $(\bar u^0, \bar F^0)$. If $(u,F,\nu)$ is a measure-valued solution of \eqref{eq:cl} with the same initial data $(\bar u^0, \bar F^0)$ then for almost all $(t,x)\in Q_T$,
\begin{equation*}
\begin{aligned}
&\nu_{t,x}=\delta_{\bar{u}, \bar{F}} \mbox{ and}\\
&(u,F)=(\bar{u}, \bar{F})\textrm{ a.e. in }(0,T)\times Q.
\end{aligned}
\end{equation*}
\end{Theorem} 

\begin{remark}
We remark that dissipative, in particular entropic, weak solutions are naturally included in the presented weak-strong uniqueness result, as in the work of Dafermos \cite{dafermos_involutions} for rank-one convex energies and entropic weak solutions with sufficiently small shocks. We refer the reader to Remark \ref{remark:section5} for a comparison between the present result and that of \cite{dafermos_involutions}. 

Further, we note that a possible extension of Theorem 1 to the wave equation \eqref{eq:intro_wave} on bounded domains would also preclude cavitation. This seems natural as the canonical example producing cavities entails energies which blow up as the local volume ratio approaches zero \cite{ball_cavitation,spector_cavitation_incompressible,spector_cavitation_general} and our growth assumptions exclude this type of behaviour.
\end{remark}

Let us introduce the {\em relative entropy} associated to system \eqref{eq:cl} and the classical solutions  $(\bar{u}, \bar{F})$:
\begin{equation*}
\eta_{\rm rel}(t,x,\lambda,\xi):= \frac12 |\lambda - \bar u|^2 + W(\xi) - W(\bar F) - S(\bar F)\cdot(\xi - \bar F)
\end{equation*}
and, similarly, at time $t=0$
\begin{equation*}
\eta^0_{\rm rel}(x,\lambda,\xi):= \frac12 |\lambda - \bar u^0|^2 + W(\xi) - W(\bar F^0) - S(\bar F^0)\cdot(\xi - \bar F^0),
\end{equation*}
where $(\bar u^0, \bar F^0)$ denotes the initial data for the classical solution $(\bar u, \bar F)$. We remark that, by Lemma \ref{lemma:technical_general} (a) in Section \ref{sec:propproof} (see also Remark \ref{rem:G}), there exists a constant $C>0$ such that for almost all $(t,x)\in Q_T$ and all $(\lambda,\xi)\in \mathbb{R}^d\times\mathbb{R}^{d\times d}$
\begin{align}\label{eq:growthetarel}
|\eta_{\rm rel}(t,x,\lambda,\xi)| &\leq C \left( |\lambda - \bar u|^2 + |V(\xi - \bar F)|^2\right)\\
&\mbox{and}\nonumber\\
\label{eq:growtheta0rel}
|\eta^0_{\rm rel}(x,\lambda,\xi)| &\leq C \left(|\lambda - \bar u^0|^2 + |V(\xi - \bar F^0)|^2\right) 
\end{align}

The proof of Theorem \ref{teo:main} is based on a variant of the relative entropy method and relies heavily on Theorem \ref{theorem:2} in Section \ref{sec:propproof}. For the ease of the reader, here we instead state a proposition, which is a simple consequence of Theorem \ref{theorem:2}, and aids the exposition of the proof of Theorem \ref{teo:main}. Its proof, along with Theorem \ref{theorem:2}, are postponed until Section \ref{sec:propproof}.
\begin{proposition}\label{prop:main}
Let $(u, F, \nu)$ be a measure-valued solution to \eqref{eq:cl} associated to the initial data $(u^0,F^0)$ as in Definition \ref{def:mvs}. In addition, let $(\bar u, \bar F)\in W^{1,\infty}(\overline{Q}_T)$ be a classical solution to \eqref{eq:cl} with initial data $(\bar u^0,\bar F^0)\in W^{1,\infty}(Q)$ where $\bar{F}^0 = \nabla \bar{y}^0$ for a zero average vector $\bar{y}^0$. Then, for some constants $C$, $C_0$, $C_1>0$, the following hold:
\begin{itemize}
\item[{\rm (a)}] for almost all $t_0\in(0,T)$ 
\begin{multline}\label{eq:ineq2}
\int_Q\langle\nu_{t_0,x},|V(\xi-\bar F(t_0,x))|^2 +  |\lambda - \bar u(t_0,x)|^2\rangle dx \\
\leq C_1\int_Q  \langle\nu_{t_0,x},\eta_{\rm rel}(t_0,x,\lambda,\xi)\rangle \, dx +  C_0 \int_Q |V(y(t_0,x)-\bar y(t_0,x))|^2dx;
\end{multline}
\item[{\rm (b)}] in addition, at the initial time $t=0$,
\begin{equation}\label{eq:eta0growth}
\int_Q \eta^0_{\rm rel}(x,u^0,F^0) dx \leq C \int_Q \Big[ |u^0 - \bar u^0|^2 +  |V(F^0-\bar F^0)|^2 \Big]dx
\end{equation}
\end{itemize}
\end{proposition}

We are ready to prove Theorem \ref{teo:main}.
\begin{proof}[Proof of Theorem \ref{teo:main}]
Let $(u, F, \nu)$ be a dissipative measure-valued solution according to Definition \ref{def:mvs} and $(\bar u,\bar F)$ a classical solution. For every $\ph\in C^1_c(Q_T,\Rd)$ and $\Phi\in C^1_c(Q_T,\Rdd)$ it holds that
\begin{align}
&\int_Q u^0\cdot\ph(0,\cdot) \,dx + \int_0^T\int_Q u \cdot \partial_t \ph \,dx dt = \int_0^T\int_Q \langle \nu_{t,x}, S\rangle\cdot\nabla\ph \,dx dt\label{eq:main1}\\
&\int_Q F^0\cdot\Phi(0,\cdot) \,dx + \int_0^T\int_Q F\cdot \partial_t \Phi \,dx dt = \int_0^T\int_Q u\cdot \dv\Phi \,dx dt\label{eq:main2},
\end{align}
where $u = \langle\nu, \pi_d\rangle$ and $F = \langle\nu,\pi_{d\times d}\rangle$. Similarly, for the classical solution, it holds that
\begin{align}
&\int_Q \bar u^0\cdot\ph(0,\cdot) dx + \int_0^T\int_Q \bar u \cdot \partial_t \ph \, dx dt = \int_0^T\int_Q  S(\bar F)\cdot\nabla\ph\, dx dt\label{eq:main3}\\
&\int_Q \bar F^0\cdot\Phi(0,\cdot) dx + \int_0^T\int_Q \bar F\cdot \partial_t \Phi \, dx dt = \int_0^T\int_Q \bar u\cdot \dv\Phi \,dx dt\label{eq:main4}.
\end{align}
Of course, by approximation, we may also take test functions which are merely Lipschitz continuous.
Then, subtracting from \eqref{eq:main1}-\eqref{eq:main2} the respective equations for the classical solutions \eqref{eq:main3}-\eqref{eq:main4} and testing with $\ph = \theta(t) \bar u$ and $\Phi = \theta(t) S(\bar F)$, where $\theta$ is a smooth function of time, we infer that
\begin{align}
&\int_Q \theta(0)\left\{(u^0- \bar u^0)\cdot \bar u(0,\cdot) + (F^0 - \bar F^0 )\cdot S(\bar F(0,\cdot))\right\}dx\notag \\
&\qquad\qquad + \int_0^T\int_Q \dot{\theta}\left\{(u - \bar u)\cdot \bar u + (F - \bar F)\cdot S(\bar F)\right\}dxdt\notag\\
&= -\int_0^T\int_Q \theta(t)\left\{(u - \bar u)\cdot \partial_t\bar u + (F - \bar F)\cdot \partial_t S(\bar F)\right\}dxdt \notag\\
&\qquad\qquad + \int_0^T\int_Q\theta(t)\left\{(u - \bar u)\cdot \dv S(\bar F) + (\langle\nu_{t,x}, S\rangle - S(\bar F))\cdot \nabla\bar u\right\}dxdt.\label{eq:main5}
\end{align}
However, the equations for the classical solution say that
\[
\partial_t \bar u  = \dv S(\bar F) \quad \mbox{and} \quad \partial_t S(\bar F)  = D(S(\bar F))\partial_t \bar F = D(S(\bar F))\nabla\bar u.
\]
Substituting back into \eqref{eq:main5}, we get
\begin{align}
&\int_Q \theta(0)\left\{(u^0- \bar u^0)\cdot \bar u(0,\cdot) + (F^0 - \bar F^0 )\cdot S(\bar F(0,\cdot))\right\}dx\notag \\
&+ \int_0^T\int_Q \dot{\theta}\left\{(u - \bar u)\cdot \bar u + (F - \bar F)\cdot S(\bar F)\right\}dxdt\notag\\
&= \int_0^T\int_Q\theta(t)\nabla\bar u\cdot\langle\nu_{t,x}, S(\xi) - S(\bar F) - D(S(\bar F))(\xi - \bar F)\rangle\,dxdt =: \mathcal{R}.\label{eq:main6}
\end{align}
We next make use of the entropy inequality which reads
\begin{equation}\label{eq:main7}
\int_Q\theta(0)\eta(u^0,F^0) + \int_0^T\int_Q \dot{\theta}\left\{\langle\nu_{t,x},\eta\rangle\,dxdt + \gamma(dxdt)\right\}\geq 0
\end{equation}
for all nonnegative functions $\theta\in C^1_c([0,T))$. For the classical solution we also have that
\begin{equation}\label{eq:main8}
\int_Q\theta(0)\eta(\bar u^0,\bar F^0) + \int_0^T\int_Q \dot{\theta}(t)\eta(\bar u,\bar F)\,dxdt = 0.
\end{equation}
%
By using the definition of relative entropy, testing with $\theta\in C^1_c([0,T))$ and integrating in space and time, we get that
\begin{align}
&\int_0^T\int_Q \dot{\theta}\left\{\langle\nu_{t,x}, \eta_{\rm rel}(t,x,\lambda,\xi)\rangle\,dxdt +\gamma(dxdt)\right\} + \int_Q \theta(0)\eta^0_{\rm rel}(x,u^0,F^0)\,dx\notag\\
&= - \int_0^T\int_Q\dot{\theta}\,\langle\nu_{t,x}, (\lambda - \bar u)\cdot\bar u + S(\bar F)\cdot(\xi - \bar F)\rangle \,dxdt \notag\\
& \qquad - \int_Q \theta(0) \left\{(u^0 - \bar u^0)\cdot\bar u^0 + S(\bar F^0)\cdot(F^0 - \bar F^0)\right\}\,dx \notag\\
& \quad + \int_0^T\int_Q\dot{\theta}\,\left\{\langle\nu_{t,x},\eta\rangle\,dxdt +\gamma(dxdt)\right\} + \int_Q \theta(0)\eta(u^0,F^0)\,dx \notag\\
&\qquad - \int_0^T\int_Q \dot\theta\, \eta(\bar u,\bar F)\,dxdt - \int_Q \theta(0) \eta(\bar u^0,\bar F^0)\,dx.\label{eq:main9}
\end{align}
By \eqref{eq:main6}, the sum of the first two integrals on the right-hand side of \eqref{eq:main9} are equal to $-\mathcal{R}$, the sum of the third and fourth is positive by \eqref{eq:main7} and the remaining integrals add up to 0 by \eqref{eq:main8}. Therefore, we conclude that
\begin{equation}\label{eq:mainetarel}
\begin{aligned}
&\int_0^T\int_Q \dot{\theta}\left\{\langle\nu_{t,x}, \eta_{\rm rel}(t,x,\lambda,\xi)\rangle\,dxdt +\gamma(dxdt)\right\}\\
&+ \int_Q \theta(0)\eta^0_{\rm rel}(x,u^0,F^0)\,dx \geq -\mathcal{R}.
\end{aligned}
\end{equation}

Next, in order to establish bounds on $\mathcal{R}$, let us consider the function
\begin{align*}
G_S(t,x,\xi) & = S(\xi) - S(\bar F) - DS(\bar F)\cdot(\xi - \bar F) \\
& = \int_0^1(1-s) D^2S(\bar F + s(\xi - \bar F))(\xi - \bar F)\cdot (\xi - \bar F)\, ds.
\end{align*}
We follow a technique similar to that in Lemma \ref{lemma:technical_general} (a). If $|\xi - \bar F|\leq 1$, we find that
\[
|G_S(t,x,\xi)|\leq |D^2S(\bar F + s(\xi - \bar F))(\xi - \bar F)\cdot (\xi - \bar F)|\leq C |\xi - \bar F|^2,
\]
where $C=C(d,\|\bar F\|_\infty)$, due to the fact that $D^2S$ is continuous (recall that $W\in C^3(\Rdd)$) and $\bar F$ is bounded. On the other hand, if $|\xi - \bar F|>1$, we infer that
\begin{align*}
|G_S(t,x,\xi)|&\leq |S(\xi) - S(\bar F)| + |DS(\bar F)||\xi - \bar F|\\
&\leq c(1+|\xi-\bar F|^{p-1})|\xi-\bar F| + c|\xi-\bar F|\leq c|V(\xi-\bar F)|^2,
\end{align*}
since $DS$ is continuous with a $(p-1)$-growth (see assumption (H3)) and $\bar F$ is bounded. This implies that there exists a constant $C>0$ such that
\begin{equation}\label{eq:growthG_S}
|G_S(t,x,\xi)|\leq C|V(\xi - \bar F)|^2.
\end{equation}
Returning to the remainder term $\mathcal{R}$, through \eqref{eq:growthG_S} and the fact that $\nabla\bar u$ is bounded, we get that
\begin{equation}\label{eq:mathcalR}
|\mathcal{R}|\leq C\int_0^T\int_Q\theta(t)\langle\nu_{t,x}, |V(\xi-\bar F)|^2\rangle\,dxdt.
\end{equation}

We are now in a position to deduce the weak-strong uniqueness. Let $(\theta_k)\subset C^1_c([0,T))$ be a bounded sequence, approximating the function
\[
\theta(\tau) = \left\{\begin{array}{cc}
1,&\tau\in[0,t)\\
(t-\tau)/\eps + 1,&\tau\in[t,t+\eps)\\
0,&\tau\in [t+\eps,T)
\end{array}\right.
\]
such that $(\theta_k)$ is nonincreasing and $\dot\theta_k(\tau)\rightarrow\dot\theta(\tau)$ for all $\tau\neq t,t+\eps$. Note that $\dot\theta_k\leq 0$ and, consequently, that $\dot\theta_k\gamma\leq 0$. Then, testing \eqref{eq:mainetarel} with $\theta_k$ we find that
\[
\int_0^T\int_Q|\dot\theta_k| \langle\nu_{\tau,x}, \eta_{\rm rel}\rangle \,dxd\tau \leq \mathcal{R} + \int_Q \eta^0_{\rm rel}(x, u^0,F^0)\,dx.
\]
However, $\nu$ is generated by a sequence with uniformly bounded energy $\eta$ which, combined with \eqref{eq:growthetarel}, implies that the functions
\[
t\mapsto \int_Q  \langle\nu_{t,x}, \eta_{\rm rel}\rangle\,dx \mbox{ and } t\mapsto \int_Q\langle\nu_{t,x}, |V(\xi-\bar F)|^2\rangle\,dx
\]
are integrable (indeed, even bounded). Then, since $\dot\theta_k$ is bounded uniformly in $k$, we may take the limit $k\to\infty$ to infer, by dominated convergence, that
\begin{equation}
\begin{aligned}
&\frac{1}{\eps}\int_t^{t+\eps}\int_Q \langle\nu_{\tau,x}, \eta_{\rm rel}\rangle\,dxd\tau \leq \int_Q \eta^0_{\rm rel}(x, u^0,F^0)\,dx\\
&+ C\int_0^{t+\eps}\int_Q \langle\nu_{\tau,x}, |V(\xi-\bar F)|^2\rangle\,dxd\tau.
\end{aligned}
\end{equation}
Then, by sending $\eps\to0$, for almost all $t\in(0,T)$, we get 
\begin{equation}\label{eq:s1}
\begin{aligned}
\int_Q \langle\nu_{t,x}, \eta_{\rm rel}\rangle\,dx&\leq \int_Q \eta^0_{\rm rel}(x, u^0,F^0)\,dx\\  
                                                                          &+C\int_0^{t}\int_Q \langle\nu_{\tau,x}, |V(\xi-\bar F)|^2\rangle\,dxd\tau.
\end{aligned}
\end{equation}
By Proposition \ref{prop:main} and the hypothesis that $\bar{u}^0=u^0$ and $\bar{F}^0=F^0$, we deduce that for almost all $t\in(0,T)$ and a suitable constant $C>0$, 
\begin{equation}\label{eq:s2}
\begin{aligned}
\int_Q\langle\nu_{t,x},|V(\xi-\bar F(t,x)|^2 &+  |\lambda - \bar u(t,x)|^2\rangle dx\\&\leq C \int_0^{t}\int_Q \langle\nu_{\tau,x}, |V(\xi-\bar F)|^2\rangle\,dxd\tau\\
&\quad +C\int_Q |V(y(t,x)-\bar y(t,x))|^2dx.
\end{aligned}
\end{equation}
In order to apply Gr\"onwall's inequality and conclude the proof, it remains to estimate the last term in the right-hand side of \eqref{eq:s2}. Note that, since $F=\nabla y$ and $\bar F = \nabla \bar{y}$, we find that for any $\phi\in C^{\infty}_{c}([0,T);C^{\infty}(Q))$
\begin{equation}\label{eq:s3}
\int_0^{T}\int_{Q}(\nabla y-\nabla\bar{y})\phi_t-(u-\bar{u})\dv\phi\,dxdt=0.
\end{equation}
Let $\psi\in C^{\infty}_{c}((0,T);C^{\infty}(Q))$ with zero spatial average and consider for any $t\in(0,T)$ the unique solution of the following elliptic problem:
\begin{equation}\label{eq:s4}
\begin{aligned}
-\Delta g(t,x)&=\psi(t,x)\\
\int_{Q}g(t,x)&=0.
\end{aligned}
\end{equation}
Then, taking $\phi=\nabla g$ in \eqref{eq:s3} we infer that 
\begin{equation}\label{eq:s5}
\int_0^{T}\int_{Q}(y-\bar{y})\psi_t+(u-\bar{u})\psi\,dxdt=0.
\end{equation}
Since $\partial_t F\in L^{\infty}(0,T;H^{-1}(Q))$ and $F=\nabla y$ it follows by the definition that $\partial_t y\in L^{\infty}(0,T;L^2(Q))$ and, integrating by parts the time derivative in \eqref{eq:s5}, we get that 
\begin{equation}\label{eq:s5bis}
\int_0^{T}\int_{Q}(y-\bar{y})_t\psi-(u-\bar{u})\psi\,dxdt=0
\end{equation}
for any $\psi\in C^{\infty}_{c}((0,T);C^{\infty}(Q))$. 
Note that this also implies the relation $\partial_t y=u$ almost everywhere in $Q_T$, giving the equivalence between system \eqref{eq:cl} and the wave equation \eqref{eq:intro_wave}.

By a straightforward density argument, we can test \eqref{eq:s5bis} with the function 
$(y-\bar{y})(1+|y-\bar{y}|^{p-2})-((y-\bar{y})|y-\bar{y}|^{p-2})_{Q}$, where we recall that $(\cdot)_{Q}$ denotes the spatial average over the cube $Q$. 
Indeed, since $d=2,3$, using the Sobolev embedding we infer that $(y-\bar{y})(1+|y-\bar{y}|^{p-2})\in L^{\infty}(0,T; L^2(Q))$, see \eqref{eq:sob} below. Then, 
\begin{equation}\label{eq:s6}
\begin{aligned}
&\frac{d}{dt}\int_Q\frac{|y(t,x)-\bar{y}(t,x)|^{2}}{2}+\frac{|y(t,x)-\bar{y}(t,x)|^{p}}{p}\,dx\\
&\leq \int_{Q}|u(t,x)-\bar{u}(t,x)|\,|y(t,x)-\bar{y}(t,x)|\,dx\\
                                                    &+\int_{Q}|u(t,x)-\bar{u}(t,x)|\,|y(t,x)-\bar{y}(t,x)|^{p-1}\,dx,
\end{aligned}
\end{equation}
and the terms involving the average $((y-\bar{y})|y-\bar{y}|^{p-2})_Q$ vanish due to the fact that $u - \bar{u}$ and $y - \bar{y}$ have zero spatial average.
Next, integrate in time and apply Young's inequality to get that for almost all $t\in(0,T)$
\begin{equation}\label{eq:e5}
\begin{aligned}
\int_Q\frac{|y(t,x)-\bar{y}(t,x)|^{2}}{2}+\frac{|y(t,x)-\bar{y}(t,x)|^{p}}{p}\,dx&\leq\int_{0}^{t}\int_{Q}|u(\tau,x)-\bar{u}(\tau,x)|^{2}\,dxd\tau\\
&+\int_{0}^{t}\int_{Q}\frac{|y(\tau,x)-\bar{y}(\tau,x)|^{2}}{2}\,dxd\tau\\
&+\int_{0}^{t}\int_{Q}\frac{|y(\tau,x)-\bar{y}(\tau,x)|^{2p-2}}{2}\,dxd\tau.
\end{aligned}
\end{equation}
Since $p\geq2$, by Sobolev embeddings, for a.e. $\tau\in(0,T)$ it holds that  
\begin{equation}\label{eq:sob}
\|y(\tau)-\bar{y}(\tau)\|_{2p-2}\leq C(\|y(\tau)-\bar{y}(\tau)\|_{p}+\|\nabla y(\tau)-\nabla\bar{y}(\tau)\|_{p}),
\end{equation}
for some constant $C$ possibly depending on $p$, $d$ and the measure of $Q$.  
Then, since $2p-2\geq p$, we have that  
\begin{equation*}
\begin{aligned}
\int_{0}^{t}\|y(\tau)-\bar{y}(\tau)\|_{2p-2}^{2p-2}\,d\tau&\leq C\int_{0}^{t}\|y(\tau)-\bar{y}(\tau)\|_{p}^{2p-2}\,d\tau
                                                                         +C\int_{0}^{t}\|\nabla y(\tau)-\nabla\bar{y}(\tau)\|_{p}^{2p-2}d\tau\\
                                                                         &\leq C\sup_{\tau\in(0,T)}\|y(\tau)-\bar{y}(\tau)\|_{p}^{p-2}\int_{0}^{t}\|y(\tau)-\bar{y}(\tau)\|_{p}^{p}d\tau\\
                                                                         &+ C\sup_{\tau\in(0,T)}\|\nabla y(\tau)-\nabla\bar{y}(\tau)\|_{p}^{p-2}\int_{0}^{t}\|\nabla y(\tau)-\nabla\bar{y}(\tau)\|_{p}^{p}\\
                                                                         &\leq C\int_{0}^{t}\|y(\tau)-\bar{y}(\tau)\|_{p}^{p}
                                                                         +C\int_{0}^{t}\|\nabla y(\tau)-\nabla\bar{y}(\tau)\|_{p}^{p},
\end{aligned}
\end{equation*}
where in the last line we used the bounds on $y$ and $\bar{y}$ in $L^{\infty}(0,T;W^{1,p}(Q))$ and the fact that they have zero average. 
Then, from \eqref{eq:e5}, we infer that for almost all $t\in(0,T)$ 
\begin{equation}\label{eq:s7}
\begin{aligned}
\int_Q\frac{|y(t,x)-\bar{y}(t,x)|^{2}}{2}+\frac{|y(t,x)-\bar{y}(t,x)|^{p}}{p}&\leq C\int_{0}^{t}\int_{Q}|u(\tau,x)-\bar{u}(\tau,x)|^{2}\\
                                                                                                           &+C\int_{0}^{t}\int_{Q}|y(\tau,x)-\bar{y}(\tau,x)|^{2}\\
                                                                                                           &+C\int_{0}^{t}\int_{Q}|y(\tau,x)-\bar{y}(\tau,x)|^{p}\\
                                                                                                           &+C\int_0^t\int_Q |F(\tau,x)-\bar{F}(\tau,x)|^{p}.
\end{aligned}
\end{equation}
By recalling the definition of $V$, Jensen's inequality now gives 
\begin{equation}\label{eq:s8}
\begin{aligned}
\int_{Q}|V(y(t,x)-\bar{y}(t,x))|^{2}\,dx\leq C\int_0^{t}\int_Q \langle\nu_{\tau,x}, |V(\xi-\bar F)|^2 + |\lambda - \bar u|^2 \rangle\,dxd\tau\\
                                                               C\int_0^{t}\int_Q|V(y(\tau,x)-\bar{y}(\tau,x))|^{2}\,dxd\tau.
\end{aligned}
\end{equation}
Adding the term 
\[
\int_{Q}|V(y(t,x)-\bar{y}(t,x))|^{2}\,dx
\]
to both sides of \eqref{eq:s2} and using \eqref{eq:s8}, 
equation \eqref{eq:s2} now reads 
\begin{align*}
&\int_Q\big[\langle\nu_{t,x},|V(\xi-\bar F)|^2 + |\lambda - \bar u|^2\rangle +|V(y - \bar y)|^2 \big]dx\\
&\leq C\int_0^t\int_Q \big[ \langle\nu_{\tau,x}, |V(\xi-\bar F)|^2 + |\lambda - \bar u|^2\rangle + |V(y - \bar y)|^2\big]\,dxd\tau,
\end{align*}
which, by Gr\"onwall's inequality, implies that the Young measure must collapse to a Dirac mass, i.e. $\nu = \delta_{(\bar u,\bar F)}$ a.e. and $y=\bar y$. Moreover, returning to \eqref{eq:main7} and using \eqref{eq:main8}, we also deduce that
\[
\int_0^T\int_Q \dot\theta\,\gamma(dxdt) \geq 0
\]
for all nonnegative $\theta\in C^1_c([0,T))$ and hence $\gamma = 0$. This concludes the proof.
\end{proof}
\section{A G\aa rding-type inequality for quasiconvex functions}\label{sec:propproof}
The proof of the weak-strong uniqueness result, Theorem \ref{teo:main}, was based on Proposition \ref{prop:main}. This proposition is a simple consequence of a more general result, Theorem \ref{theorem:2} below, which forms the main part of Section \ref{sec:propproof}. We note that Theorem \ref{theorem:2} is independent of the equations and it is of interest in its own right. It essentially states that, on smooth maps, quasiconvexity behaves like an integral version of convexity and it is the result which allows us to adapt the relative entropy method to the quasiconvex setting.

We denote by 
\[
\mathcal{F}_K:=\{H\in W^{1,\infty}(\overline{Q},\Rdd)\,:\, \|H\|_{W^{1,\infty}}\leq K\},
\]
and by $C(f,K)$ a positive constant that depends only on the $L^\infty$ bounds of a function $f$ or any of its derivatives in a ball determined by $K$. Next, for $f:\Rdd\to\R$, we define the function $G_f: \Rdd\times \Rdd\to\RR$ by
\begin{equation}\label{eq:g}
G_f(z,\xi):=f(z + \xi) - f(z) - Df(z):\xi =\int_0^1 (1-s)D^2f(z+s\xi)\xi:\xi\,ds.
\end{equation}


\begin{Theorem}\label{theorem:2}
Assume that $W\in C^2(\Rdd)$ is strongly quasiconvex and satisfies a $p$-coercivity and growth, i.e. for all $\xi\in\Rdd$
\[
c(-1 + |\xi|^p) \leq W(\xi)\leq c(1+|\xi|^p).
\] 
In addition, let $\nu^F = (\nu^F_{t,x})_{(t,x)\in Q_T}$ be a family of probability measures generated by a sequence of spatial gradients $(\nabla y^k)$ such that
\begin{align*}
&(y^k)\mbox{ is bounded in } L^\infty(0,T; W^{1,p}(Q))\cap L^{\infty}(0,T; H^{1}_{0}(Q)) \\
&(\partial_t \nabla y^k)\mbox{ is bounded in } L^\infty(0,T;H^{-1}(Q))
\end{align*}
and write $\nabla y = \langle\nu^F,{\rm id}\rangle$ for its centre of mass.
Then, for almost all $t_0\in(0,T)$,
\begin{multline}\label{eq:step1bis}
\int_Q\langle\nu^F_{t_0,x},|V(\xi-\bar F(t_0,x)|^2 \rangle dx \\
\leq C_1\int_Q  \langle\nu^F_{t_0,x},G_W(\bar F(t_0,x), \xi - \bar F(t_0,x))\rangle \, dx +  C_0 \int_Q |V(y(t_0,x)-\bar y(t_0,x))|^2dx.
\end{multline}
\end{Theorem}

\begin{remark}\label{remark:section5}
\begin{enumerate}
\item Inequality \eqref{eq:step1bis} can be seen as a G\aa rding-type inequality for quasiconvex functions. In fact, it should be contrasted with \cite[Lemma 4.3]{dafermos_involutions} where a similar inequality is established in the case of (strongly) rank-one convex functions. The crucial difference here is that, unlike \cite{dafermos_involutions}, there is no need to assume a condition of small local oscillations. In \cite{dafermos_involutions}, the need for this assumption arises when `delocalising' the strong ellipticity condition, i.e. rank-one convexity, from a fixed $x\in\mathbb{R}^d$ to a cube in $\mathbb{R}^d$. In the present work, we also need to delocalise the strong quasiconvexity condition in the same way, however, we are able to achieve this through a strategy developed by Kristensen and Campos Cordero, see \cite{JCCVMO} and also Campos Cordero and Koumatos \cite{JCCKK} using an idea of K. Zhang \cite{Zhang92} showing that smooth extremals of strongly quasiconvex energies are minimisers with respect to spatially localised variations.

\item We point out that the assumption on the time derivative $(\partial_t \nabla y^k)$ being bounded in\\ $L^{\infty}(0,T;H^{-1}(Q))$ is used in order to infer the strong convergence of $(y^k)$ in $L^{p}((0,T)\times Q)$ and obtain Lemma \ref{lemma:decomposition} (3) which is crucial. Theorem \ref{theorem:2} can equivalently be stated under the assumption that $y^k\to y$ in $L^{p}((0,T)\times Q)$.
\item Note the relaxed assumptions on $W$. In particular, there is no need to invoke the assumed growth of $D^2W$. This is only required in order to bound the term $\mathcal{R}$ in the proof of Theorem \ref{teo:main}. The same remark goes for the regularity of $W$ which may be assumed to be $C^2$ throughout Section \ref{sec:propproof}.
\end{enumerate}
\end{remark}
We are now in a position to prove Proposition \ref{prop:main}.

\begin{proof}[Proof of Proposition \ref{prop:main}]
Note that the proof of part (b) is an immediate consequence of Lemma \ref{lemma:technical_general} part (a) below (see also Remark \ref{rem:G}).
As for part (a), let $(u,F,\nu)$ be a measure-valued solution as in Definition \ref{def:mvs} and let $(u^k,F^k)$ be a generating sequence, where $F^k = \nabla y^k$ must hence satisfy the assumptions of Theorem \ref{theorem:2}.

Simply note that whenever $g:Q_T\times\Rd\times \Rdd \to \RR$ is a function that admits an additive decomposition
\[
g(t,x,\lambda,\xi) = g_d(t,x,\lambda) + g_{d\times d}(t,x,\xi),
\]
where 
\[
|g_d|\leq c(1+|\cdot |^2)\mbox{ and } |g_{d\times d}|\leq c(1+|\cdot |^p),
\]
the action of $\nu$ is equivalent to the action of $\nu^u\otimes \nu^F$ where $\nu^u$, $\nu^F$ are the measures generated by $(u^k)$ and $(F^k)$ respectively. Hence, it suffices to prove \eqref{eq:step1bis} and then simply add the term
\[
\int_Q \langle\nu^u_{t_0,x}, |\lambda - \bar u(t_0,x)|^2\rangle \,dx
\]
to conclude the proof of Proposition \ref{prop:main}.
\end{proof}

Next, we present a series of lemmas which are crucial for the proof of Theorem \ref{theorem:2}. Lemma \ref{lemma:technical_general} below provides some crucial properties of $G$. The proof of (a) and (b), originating in \cite{AcFus87}, is based on \cite{JCCVMO,JCCKK}, whereas (c) on \cite{GM}. 

\begin{lemma}\label{lemma:technical_general}
Let $f\in C^2(\Rdd)$ such that
\[
|f(\xi)|\leq C(1+|\xi|^p)\mbox{ and }|Df(\xi)|\leq C(1+|\xi|^{p-1}).
\]
%
%
Then, the following hold:

\begin{itemize}
\item[(a)] There exists $C=C(f,K)$ such that for all $z\in \overline{B(0,K)}$, $\xi_1,\xi_2\in\Rdd$
\[
|G_f(z,\xi_1) - G_f(z,\xi_2)| \leq C (|\xi_1| + |\xi_2| + |\xi_1|^{p-1}+|\xi_2|^{p-1})|\xi_1 - \xi_2|.
\]
In particular, $|G_f(z,\xi)| \leq C|V(\xi)|^2$.
\item[(b)] For every $\delta>0$ there exists $R = R(\delta,f,K)>0$ such that for all $z_1,\,z_2\in \overline{B(0,K)}$ with $|z_1-z_2|<R$, it holds that
\[
|G_f(z_1,\xi) - G_f(z_2,\xi)| \leq \delta |V(\xi)|^2.
\]
\item[(c)] If in addition $f(\xi)\geq -d+c|\xi|^p$, there exist constants $C=C(f,K)$, $\tilde{C}=\tilde{C}(f,K)$ such that for all $z\in \overline{B(0,K)}$
\[
G_f(z,\xi) \geq C |\xi|^p - \tilde{C}|\xi|^2.
\]
\item[(d)] If $f$ is also strongly convex, i.e. $D^2f(z) \xi:\xi\geq \gamma |\xi|^2$, then there exists $C=C(f,K)$ such that for all $z\in \overline{B(0,K)}$,
\[
G_f(z,\xi) \geq C|V(\xi)|^2.
\]
\end{itemize}
\end{lemma}

\begin{proof}
For the proof of (a), if $|\xi_1| + |\xi_2| \leq 1$ and 
for all $|z|\leq K$,
\begin{align*}
&|G_f(z,\xi_1) - G_f(z,\xi_2)| \leq \int_0^1 \left| D^2f(z + s \xi_1)\xi_1:(\xi_1 - \xi_2) \right| \\
&+\int_0^1 \left| \left[D^2f(z + s \xi_1) - D^2f(z + s \xi_2)\right]\xi_1:\xi_2 \right|+\int_0^1 \left| D^2f(z + s \xi_2)\xi_2:(\xi_1 - \xi_2) \right|  \\
&\quad\quad \leq  C(f,K)  (|\xi_1| + |\xi_2|)|\xi_1 - \xi_2|.
\end{align*}

On the other hand, if  $|\xi_1| + |\xi_2|>1$, note that
\begin{align*}
G_f(z,\xi_1) - G_f(z,\xi_2) & = \int_0^1 Df(z + \xi_2 + s(\xi_1-\xi_2)):(\xi_1 - \xi_2)\,ds - Df(z):(\xi_1 - \xi_2).
\end{align*}
We may thus estimate
\begin{align*}
\left|G_f(z,\xi_1) - G_f(z,\xi_2) \right| & \leq c(1 + |z|^{p-1} + |\xi_1|^{p-1} + |\xi_2|^{p-1})|\xi_1 - \xi_2| \\
&+ c(1+|z|^{p-1})|\xi_1 - \xi_2|\\
&\leq C(K)( |\xi_1| + |\xi_2| + |\xi_1|^{p-1} + |\xi_2|^{p-1})|\xi_1 - \xi_2|,
\end{align*}
since $|\xi_1| + |\xi_2| > 1$ and $|z|\leq K$. This completes the proof of (a).

Concerning (b), again we split into two cases. If $|\xi|\leq 1$, by \eqref{eq:g}
\begin{equation*}
|G_f(z_1,\xi) - G_f(z_2,\xi)|
\leq C(f,K) \left| z_1-z_2\right||\xi|^2,
\end{equation*}
 whereas, if $|\xi|>1$,
\begin{align*}
|G_f(z_1,\xi) - G_f(z_2,\xi)|& \leq \left|f(z_1 + \xi) - f(z_2 + \xi) \right| + \left|f(z_1) - f(z_2) \right| \\
 & \quad +\left|Df(z_1) - Df(z_2) \right||\xi|\leq C(f,K) |z_1- z_2| |V(\xi)|^2.
\end{align*}
Hence, given $\delta>0$ we may choose $R \leq \delta/ C(f,K)$.

Regarding (c), we follow \cite[Section 3.2]{GM}. If $|\xi| \leq 1$, we can find $C=C(f,K)>0$ such that
\begin{align*}
G_f(z,\xi) & = \int_0^1(1-s) D^2 f(z + s \xi)\,ds\, \xi : \xi \geq - C |\xi|^2 \geq |\xi|^p - (C+1)|\xi|^2.
\end{align*}
On the other hand, if $|\xi|> 1$, by coercivity, we get
\[
G_f(z,\xi) \geq d_1|\xi|^p - d_2(f,K) - d_3(f,K)|\xi| \geq d_1|\xi|^p - (d_2+d_3)|\xi|^2,
\]
concluding the proof of (c).

For the proof of (d), by Young's inequality,
\begin{align*}
G_f(z,\xi) 
 & \geq c|\xi|^p - C(f,K,\delta) -  \delta |\xi|^p  \geq \tilde{c}|\xi|^p - C(f,K,\delta),
\end{align*}
for $\delta$ small enough. Hence, if $|\xi|^p \geq 2C(f,K,\delta)/\tilde{c} + 1:= R^p$, we deduce that
\[
G_f(z,\xi) \geq \frac{\tilde{c}}{2}|\xi|^p \geq \frac{\tilde{c}}{4}|V(\xi)|^2,
\]
as $|\xi|\geq 1$. On the other hand, for $|\xi| < R$, the strong convexity gives
\begin{align*}
G_f(z,\xi) 
& \geq \frac{1}{2}\gamma |\xi|^2 \geq \frac{1}{4}\gamma |\xi|^2 + \frac{R^2}{4}\gamma \frac{|\xi|^2}{R^2} \geq  \frac{1}{4}\gamma |\xi|^2 + \frac{R^2}{4R^p}\gamma |\xi|^p\geq \tilde{c} |V(\xi)|^2.
\end{align*}
Combining the two cases, we infer the result.
\end{proof}

\begin{remark}\label{rem:G}
 Letting $f=W$, $\xi_2 = 0$, $z = \bar F(t,x)$ and $\xi_1 = \xi - \bar F(t,x)$, Lemma \ref{lemma:technical_general} (a) implies that for almost all $(t,x)\in Q_T$ and all $(\lambda,\xi)\in \mathbb{R}^d\times\mathbb{R}^{d\times d}$,
\[
|\eta_{\rm rel}(t,x,\lambda,\xi)| \leq C \left( |\lambda - \bar u|^2 + |V(\xi - \bar F)|^2\right).
\]
Similarly, $|\eta^0_{\rm rel}(x,\lambda,\xi)| \leq C \left(|\lambda - \bar u^0|^2 + |V(\xi - \bar F^0)|^2\right)$.
\end{remark}

Next, we present another simple, yet crucial, lemma.

\begin{lemma}
\label{lemma:tWall}
There exists a constant $c_2 = c_2(W,K)$ such that the function
\[
\tW(\xi) : = W(\xi) -  c_2 |V(\xi)|^2
\]
is $p$-coercive and satisfies the following:
\begin{itemize}
\item[(a)] $\tW$ is strongly quasiconvex with constant $c_0/2$ at all $\xi\in \overline{B(0,K)}$, i.e.
\[
\int_Q\tW(\xi+\nphi) - \tW(\xi) \geq \frac{c_0}{2} \int_Q |V(\nphi)|^2,\quad\forall\,|\xi|<K,\forall\,\varphi\in W^{1,p}(Q).
\]
\item[(b)] For all $\xi\in \overline{B(0,K)}$ and all $Q^\prime\subset Q$ it holds that
\begin{equation}\label{eq:qc_rc}
\int_{Q^\prime} D^2\tW(\xi) \nphi:\nphi \geq c_0 \int_{Q^\prime} |\nphi|^2\quad\forall\,\varphi\in W^{1,p}_0(Q^\prime).
\end{equation}
\end{itemize}
\end{lemma}

\begin{proof}
The $p$-coercivity of $\tW$ follows from that of $W$. Part (a) follows by applying Lemma \ref{lemma:technical_general} (a) to the function $f(\xi) = |V(\xi)|^2$, while (b) by viewing quasiconvexity as a minimality condition and considering the second variation.
\end{proof}

The following result which can be viewed as a G\aa rding inequality itself is inspired by Dafermos \cite[Lemma 4.3]{dafermos_involutions}.

\begin{proposition}\label{prop:1}
There exists $c_1 = c_1(W,K)>0$ such that for any $H\in\mathcal{F}_K$
\[
\int_{Q} D^2\tW(H(x))\nphi:\nphi \geq \frac{c_0}{2} \int_Q |\nphi|^2 - c_1 \int_Q|\varphi|^2,\quad\forall \varphi\in W^{1,2}(Q).
\]
\end{proposition}

\begin{proof}
Fix $\delta>0$ and a finite cover $\{Q_i\}\subset Q$, $Q_i=Q_i(x_i,r_i)$, such that
\[
|D^2\tW(H(x)) - D^2\tW(H(x_i))| \leq c_0 \delta(1-\delta)^2.
\]
Since $H\in\mathcal{F}_K$ and $\tW\in C^2(\Rdd)$, the cover can be chosen uniformly for $H\in\mathcal{F}_K$.
Next, choose a partition of unity $\{\rho_i\}$ subordinate to the cover $\{Q_i\}$ such that ${\rm supp}\,\rho_i\subset Q_i$ and $\sum_i \rho_i^2 = 1$.
Given $\varphi\in W^{1,2}(Q)$, we find that for all $H\in\mathcal{F}_K$,
\begin{align}\label{eq:1}
\int_{Q} D^2\tW(H(x))\nphi:\nphi&= \sum_i \int_{Q_i}\!\!\rho_i^2 D^2\tW(H(x_i))\nphi:\nphi\nonumber\\
& + \sum_i \int_{Q_i}\!\!\rho_i^2 \left[D^2\tW(H(x)) - D^2\tW(H(x_i))\right]\nphi:\nphi \nonumber\\
& \geq \sum_i \int_{Q_i}\!\!\! D^2\tW(H(x_i))(\rho_i\nphi)\!:\!(\rho_i\nphi)
- c_0\delta(1-\delta)^2\!\!\! \int_{Q}|\nphi|^2. 
\end{align}
Note that $\rho_i\nphi = \nabla(\rho_i\varphi) - \varphi\otimes \nabla\rho_i$ with
$\rho_i\varphi \in W^{1,2}_0(Q_i)$ and $|H(x_i)|\leq K$. Then, by \eqref{eq:qc_rc} and Young's inequality, we infer that
%
\begin{align}\label{eq:3}
\int_{Q_i}\rho_i^2 D^2\tW(H(x_i))\nphi:\nphi 
&\geq c_0(1-\delta) \int_{Q_i}|\nabla(\rho_i\varphi)|^2 - C\int_{Q_i} |\varphi|^2,
\end{align}
where $C=C(\tW,K,\delta)$.
Through Young's inequality we also find that
\[
\int_{Q_i}\rho_i^2 D^2\tW(H(x_i))\nphi:\nphi \geq c_0(1-\delta)^2\int_{Q_i} \rho_i^2|\nabla\varphi|^2 - C(\delta)\int_{Q_i}|\varphi|^2,
\]
where $C(\delta)$ also depends on $\|\nabla\rho_i\|_{\infty}$, in turn depending only on $\delta$ and $W$.
Then, after summing up, \eqref{eq:1} results in
\begin{align*}
\int_{Q} D^2\tW(H(x))\nphi:\nphi 
& \geq c_0(1-\delta)^3\int_{Q} |\nabla\varphi|^2 - C(\delta)\int_{Q}|\varphi|^2.
\end{align*}
To conclude the proof, fix $\delta = 1-2^{-1/3}$ and rename $C=C(W,K)=:c_1$.
\end{proof}

We next present Proposition \ref{prop:2} which is used repeatedly in the proof of Theorem \ref{theorem:2}. We note that, for $C^2$ functions with a $p$-growth, this is an equivalent characterisation of strong quasiconvexity.

\begin{proposition}\label{prop:2}
Let $\left(H_k\right)\subset \mathcal{F}_K$, $(h_k)\subset W^{1,p}(Q)$, $(a_k)\subset \R$ such that
\begin{equation*}
\begin{aligned}
&a_k^{-1}V(h_k) \rightarrow 0\mbox{ strongly in }L^2(Q),
&\left(a_k^{-1}V(\nabla h_k\right))\mbox{ is bounded in }L^2(Q).
\end{aligned}
\end{equation*}
Then,
\[
\liminf_{k\to\infty} \frac{c_0}{4} a_k^{-2}\int_Q |V(\nabla h_k)|^2 \leq \liminf_{k\to\infty} a_k^{-2}\int_Q \tGF(H_k(x),\nabla h_k).
\]
\end{proposition}

\begin{proof}
The proof is identical to \cite[Proposition 4.6]{JCCKK}, noting that there is no dependence on the lower order terms $(h_k)$ in $W$ and thus no assumptions on $(h_k)$ are required. Here, we repeat the argument for completeness.
Letting $\delta = c_0/4$ in Lemma \ref{lemma:technical_general} (b), we find $R=R(c_0,\tW,K)$ such that for all $H\in\mathcal{F}_K$ and whenever $|x-x_0|<R$
\[
\tGF(H(x),\xi) \geq \tGF(H(x_0),\xi) - \frac{c_0}{4} |V(\xi)|^2.
\]
In particular, for $\varphi\in W^{1,p}_0(Q(x_0,R))$,
\begin{align}\label{eq:Zhang}
\int_{Q(x_0,R)} \tGF(H(x),\nphi) 
 &\geq \frac{c_0}{4}\int_{Q(x_0,R)}|V(\nphi)|^2,
\end{align}
by Lemma \ref{lemma:tWall} (b), i.e. that $\tW$ is strongly quasiconvex, and the fact that $\int_{Q(x_0,R)} D\tW(H(x_0)):\nphi = 0$.
Next, note that since $\left(a_k^{-1}V(\nabla h_k)\right)$ is bounded in $L^2(Q)$ we may assume that
\[
a_k^{-2} |V(\nabla h_k)|^2 \mathcal{L}^d\llcorner Q \overset{\ast}\rightharpoonup \mu,\quad\mbox{in }\mathcal{M}(\overline{Q}) = \left(C(\overline{Q})\right)^\ast.
\]
Since $\mu$ is a positive measure, 
we can find a finite cover of $Q$ by cubes $Q(x_j,r_j)$ with the property that $r_j<R$, so that \eqref{eq:Zhang} applies, and that
\begin{equation}\label{eq:muzero}
\mu(\overline{Q}\cap \partial Q(x_j,r_j)) = 0.
\end{equation}
We consider cut-off functions $\rho_j \in C^\infty_c(Q(x_j,r_j))$ such that for $\lambda\in(0,1)$, $\rho_j \equiv 1$ in $Q(x_j,\lambda r_j)$, $\rho_j \equiv 0$, on $\partial Q(x_j,r_j)$ and $\|\nabla\rho_j\|_{L^{\infty}(Q)}\leq C/(1-\lambda)$.
By \eqref{eq:Zhang}, for $\varphi\in W^{1,p}(Q)$ and thus $\rho_j\varphi\in W^{1,p}_0(Q(x_j,r_j))$, we find that
\begin{multline*}
\frac{c_0}{4}\int_{Q(x_j,\lambda r_j)}|V(\nphi)|^2 + \frac{c_0}{4}\int_{Q(x_j,r_j)\setminus Q(x_j,\lambda r_j)}|V(\nabla(\rho_j\varphi))|^2 \\
\leq \int_{Q(x_j,\lambda r_j)} \tGF(H(x),\nphi)+ \int_{Q(x_j,r_j)\setminus Q(x_j,\lambda r_j)}\tGF(H(x),\nabla(\rho_i\varphi)) \\
\leq \int_{Q(x_j,\lambda r_j)} \tGF(H(x),\nphi)+ C\int_{Q(x_j,r_j)\setminus Q(x_j,\lambda r_j)}|V(\nabla(\rho_i\varphi))|^2
\end{multline*}
where by Lemma \ref{lemma:technical_general} (a), $C=C(\tW,K)$. Using Lemma \ref{lemma:technical_general} (a) and summing over $j$, after setting 
$\varphi = h_k$, $H = H_k$, dividing by $a_k^{-2}$ and taking the liminf, we obtain
\begin{multline*}
\liminf_{k\to\infty} \frac{c_0}4 a_k^{-2} \int_Q |V(\nabla h_k)|^2 \leq \liminf_{k\to\infty} a_k^{-2} \int_Q\tGF( H_k(x),\nabla h_k)\\
+ C\limsup_{k\to\infty} \sum_j \int_{Q(x_j,r_j)\setminus Q(x_j,\lambda r_j)} a_k^{-2}|V(\nabla h_k)|^2 + a_k^{-2}\left| V\left(\frac{h_k}{1-\lambda}\right)\right|^2.
\end{multline*}
However, $a_k^{-1}V(h_k)\rightarrow 0$ in $L^2(Q)$ and $a_k^{-2} |V(\nabla h_k)|^2 \mathcal{L}^d\llcorner Q \overset{\ast}\rightharpoonup \mu$ in $\mathcal{M}(Q)$, so that
\begin{multline*}
\liminf_{k\to\infty} \frac{c_0}4 a_k^{-2} \int_Q |V(\nabla h_k)|^2 \leq \liminf_{k\to\infty} a_k^{-2} \int_Q\tGF( H_k(x),\nabla h_k)\\
+ C\sum_j \mu\left(\overline{Q}\cap\left(\overline{Q(x_j,r_j)}\setminus Q(x_j,\lambda r_j)\right)\right)
\end{multline*}
Taking the limit $\lambda\to 1$ and noting \eqref{eq:muzero}, we conclude the proof.
\end{proof}

Next we present a lemma which is frequently used. This is a simple observation which can be seen as a restatement of the continuity of translations; see also \cite{KP91}.

\begin{lemma}\label{lemma:translations}
Let $v\in L^\infty(0,T; L^p(Q))$ for any $p\in[1,\infty)$. Then, up to a subsequence which is not relabelled and for almost all $t_0\in(0,T)$, it holds that
\[
\lim_{\eps\to0}\int_{Q_T} |v(t_0+\eps t/T,x) - v(t_0,x)|^p\,dxdt = 0.
\]
\end{lemma}

\begin{proof}
Consider $t_0$ as a variable and integrate in time twice to infer that
\begin{align*}
&\int_0^T \int_{Q_T} |v(t_0+\eps t/T,x) - v(t_0,x)|^p \,dxdtdt_0\\
\quad & = \int_0^T \int_{Q_T} |v(t_0+\eps t/T,x) - v(t_0,x)|^p \,dxdt_0dt \\
& = \int_0^T \|v(\cdot +\eps t/T,\cdot) - v(\cdot,\cdot)\|^p_{L^p(Q_T)}\,dt.
\end{align*}
However, by the continuity property of translations, for almost all $t$,
\[
\|v(\cdot +\eps t/T,\cdot) - v(\cdot,\cdot)\|^p_{L^p(Q_T)}\to 0,\mbox{ as }\eps\to0.
\]
Also, since $v\in L^\infty(0,T; L^p(Q))$, the above quantity is also bounded uniformly in $\eps$ so that, by dominated convergence,
\[
\lim_{\eps\to0} \int_0^T \int_{Q_T} |v(t_0+\eps t/T,x) - v(t_0,x)|^p \,dxdtdt_0 = 0.
\]
In particular, up to a subsequence (not relabelled), for almost all $t_0$
\[
\lim_{\eps\to0} \int_{Q_T} |v(t_0+\eps t/T,x) - v(t_0,x)|^p = 0.
\]
\end{proof}

Let us note that, to prove Theorem \ref{theorem:2}, we are required to localise our measure-valued solution in time, i.e. consider the measures $(\nu_{t_0,x})_{x\in Q}$.
As it is perhaps evident, particularly after Lemma \ref{lemma:translations}, the generating sequences for these measures will be given by a time scaling of the generating sequence
for $\nu=(\nu_{t,x})_{(t,x)\in Q_T}$. However, the lack of equiintegrability of the assumed generating sequence presents an obstacle and, here, we present a final lemma which assures that an equiintegrable generating sequence of spatial gradients $(\nabla z^k)$ can be chosen which has the additional property that $(z^k)$ converges strongly to $y(t_0,x)$ in $L^p(Q_T)$. This can be seen as a time-dependent generalisation of the celebrated decomposition theorem of Kristensen \cite{kristensen99}. At this stage, we remark that if instead of measure-valued solutions weak solutions are to be considered, no decomposition is required and the proof of Theorem \ref{theorem:2} simplifies significantly. 

\begin{lemma}\label{lemma:decomposition}
Let $\nu^F = (\nu^F_{t,x})_{(t,x)\in Q_T}$ be a family of probability measures as in Theorem \ref{theorem:2}.
Then, for almost all $t_0\in(0,T)$, there exists a sequence of spatial gradients $(\nabla z^k)$ also bounded in $L^\infty(0,T;L^p(Q))$, in particular $z_k\in L^\infty(0,T; W^{1,p}(Q))\cap L^{\infty}(0,T; H^{1}_{0}(Q))$, with the following properties:
\begin{enumerate}
\item[(1)] $(\nabla z^k)$ generates the measure $(\nu^F_{t_0,x})_{x\in Q}$ as a $p$-Young measure;
\item[(2)] $(|\nabla z^k|^p)$ is weakly relatively compact in $L^1(Q_T)$;
\item[(3)] $z^k \rightarrow y(t_0,\cdot)$ strongly in $L^p(Q_T)$.
\end{enumerate}
\end{lemma}
\begin{proof}
For $t_0\in(0,T)$ define
\[
y^{k,\eps}(t,x) := y^k(t_0+\eps t/T,x).
\]
We claim that for a.e. $t_0$ an appropriate subsequence of $(\eps_k)$ can be chosen such that $(\nabla y^{k,\eps_k})$ generates the measure $(\nu^F_{t_0,x})_{x\in Q}$ and that $y^{k,\eps_k}\rightarrow y(t_0,\cdot)$ in $L^p(Q_T)$. To this end, note that, up to a subsequence which is not relabelled, for any $g\in C_p(\mathbb{R}^{d\times d})$ and any Borel set $E\subset Q_T$ for a.e. $t_0\in (0,T)$ it holds that
\begin{equation}\label{eq:step2_1}
\lim_{\eps\to 0}\int_E |\langle\nu^F_{t_0+\eps t/T,x},g\rangle - \langle\nu^F_{t_0,x},g\rangle| = 0.
\end{equation} 
This is a consequence of Lemma \ref{lemma:translations} noting that the function $v(t,x) = \langle\nu^F_{t,x},g\rangle$ is an element of $L^\infty(0,T; L^1(Q))$ since, by Lemma \ref{lemma:supboundint},
\[
\sup_t \int_Q \langle\nu^F_{t,x},|\cdot |^p\rangle < \infty.
\]
Hence, it follows that for any such $g$ and $E$, denoting by $\chi_E$ the characteristic function of $E$ and $t_0$ fixed a.e. in $(0,T)$ using \eqref{eq:step2_1}, we infer that
\begin{align}
\lim_{\eps\to0}\lim_{k\to\infty}\int_{E} g(\nabla y^{k,\eps}(t,x)) & = \lim_{\eps\to0} \lim_{k\to\infty} \frac{1}{\eps}\int_{t_0}^{t_0+\eps}\int_Q\chi_E((t-t_0)T/\eps,x)g(\nabla y^k(t,x))\notag\\
& = \lim_{\eps\to0}\int_{Q_T}\chi_E(t,x)\langle\nu^F_{t_0+\eps t/T,x},g\rangle\notag\\
& = \int_E \langle\nu^F_{t_0,x},g\rangle.\label{eq:step2_11}
\end{align} 

In addition, we find that
\begin{align}\label{eq:I+II}
&\int_{Q_T}|y^k(t_0+\eps t/T,x) - y(t_0,x)|^p \notag \\
& \leq C\int_{Q_T}|y^k(t_0+\eps t/T,x) - y(t_0+\eps t/T,x)|^p\notag\\
&+ C\int_{Q_T}|y(t_0+\eps t/T,x) - y(t_0,x)|^p =: I + II,
\end{align}
where $\eps$ denotes the (non-relabelled) subsequence chosen in \eqref{eq:step2_1}. Noting that $y\in L^\infty(0,T; L^p(Q))$, Lemma \ref{lemma:translations} says that, up to extracting a further subsequence, for a.e. $t_0\in (0,T)$
\[
\lim_{\eps\to0} II = 0.
\]
Regarding term $I$, note that
\[
(y^k) \subset L^\infty(0,T; W^{1,p}(Q))\mbox{ and } (\partial_t y^k )\subset L^2(0,T; L^2(Q))
\]
are both bounded in the respective spaces. Then, since $W^{1,p}(Q) \subset\subset L^p(Q) \subset L^2(Q)$, the Aubin--Lions lemma says that
\[
y^k \rightarrow y \mbox{ in } C(0,T; L^p(Q)),
\]
i.e.
\begin{align*}
\lim_{k\to\infty} I & = \lim_{k\to\infty} \frac{1}{\eps}\int_{t_0}^{t_0+\eps} \int_Q |y^k(t,x) - y(t,x)|^p\,dx dt \\
& \leq \lim_{k\to\infty}\sup_t \int_Q |y^k(t,x) - y(t,x)|^p\,dx = 0.
\end{align*}
Returning to \eqref{eq:I+II}, we infer that for a.e $t_0\in (0,T)$
\begin{equation}\label{eq:step2_13}
\lim_{\eps\to0}\lim_{k\to\infty} \int_{Q_T}|y^k(t_0+\eps t/T,x) - y(t_0,x)|^p = 0.
\end{equation}
Now, for $g$ and $E$ in a countable dense subset of $C_p(\Rdd)$ and of the collection of Borel subsets of $Q_T$, respectively, we may choose a subsequence $(\eps_k)$ such that \eqref{eq:step2_11} and \eqref{eq:step2_13} hold. In particular, for $t_0$ fixed almost everywhere in $(0,T)$,
\[
\lim_{k\to\infty}\int_{E} g(\nabla y^{k,\eps_k}(t,x)) = \int_E \langle\nu^F_{t_0,x},g\rangle
\]
for all elements of the countable subsets where $g$ and $E$ belong and, by density, for all $g\in C_p(\Rdd)$ and all $E\subset Q_T$, i.e.
\[
g(\nabla y^{k,\eps_k}) \rightharpoonup \langle\nu^F_{t_0,x},g\rangle \mbox{ in }L^1(Q_T)
\]
and $(\nabla y^{k,\eps_k})$ generates the measure $(\nu_{t_0,x})_x$. Note also that
\[
(\nabla y^{k,\eps_k})\subset L^\infty(0,T;L^p(Q)).
\]

Next, we perform a suitable decomposition of $(\nabla y^{k,\eps_k})$ to infer the existence of the required sequence $(\nabla z^k)$. 
For $n\in \mathbb{N}$ consider the truncation operator
\[
T_n(\xi) = \left\{\begin{array}{ll}\xi,& |\xi| \leq n\\
 n\xi/|\xi|,& |\xi|>n. 
\end{array}\right.
\]
We infer that
\begin{equation}\label{eq:step2_2}
\lim_{n\to\infty}\lim_{k\to\infty}\int_{Q_T}|T_n(\nabla y^{k,\eps_k})|^p = \lim_{n\to\infty}\int_{Q_T}\langle\nu_{t_0,x}, |T_n(\cdot)|^p\rangle = \int_{Q_T}\langle\nu_{t_0,x}, |\cdot|^p\rangle,
\end{equation}
where the second equality follows from monotone convergence.

Moreover, note that
\begin{equation}\label{eq:step2_3}
\lim_{n\to\infty}\lim_{k\to\infty}\int_{Q_T}|T_n(\nabla y^{k,\eps_k}) - \nabla y^{k,\eps_k}|  \leq \lim_{n\to\infty}\sup_{k}\int_{\{|\nabla y^{k,\eps_k}|>n\}} 2|\nabla y^{k,\eps_k}| \to 0,
\end{equation}
due to the equiintegrability of ($\nabla y^{k,\eps_k}$). Then, there exists a subsequence $(k_n)$, such that the functions
\[
V_n(t,x):=T_n(\nabla y^{k_n,\eps_{k_n}}(t,x))
\]
simultaneously satisfy
\begin{align}\label{eq:step2_3a}
&\lim_{n\to\infty}\int_{Q_T}|V_n - \nabla y^{k_n,\eps_{k_n}}|  = 0,\\
\label{eq:step2_2a}
&\lim_{n\to\infty}\int_{Q_T}|V_n|^p =  \int_{Q_T}\langle\nu_{t_0,x}, |\cdot|^p\rangle.
\end{align}
In particular, due to the bounds on $(y^{k,\eps_k})$, $(V_n)$ is bounded in $L^\infty(0,T;L^p(Q))$. Also, by \eqref{eq:step2_3a}, $(V_n)$ generates the measure $(\nu_{t_0,x})_x$ and, by  \eqref{eq:step2_2a}, $(|V_n|^p)$ must be weakly relatively compact in $L^1(Q_T)$.

Next, for almost all $t$, consider the Hodge decomposition of $V_n(t,\cdot) \in L^p(Q)$, that is
\begin{equation}\label{eq:hel}
V_n(t,\cdot) = \mathcal P^{\rm curl}(V_n(t,\cdot)) + \mathcal P^{\rm div}(V_n(t,\cdot)) = : \nabla z^n(t,\cdot) + g^n(t,\cdot) 
\end{equation}
where $\mathcal P^{\rm curl}$, $\mathcal P^{\rm div}$ denote respectively the projections onto the space of curl-free and divergence-free vector fields in $L^p(Q)$. In particular, we may also assume that $z^n(t,\cdot)\in W^{1,p}(Q)$ has zero average. We claim that $(\nabla z^n)$ is the required sequence. 

For convenience, let us write $y^n = y^{k_n,\eps_{k_n}}$ and recall that $\mathcal P^{\rm curl}$ is a strong $(r,r)$ operator for any $1<r<\infty$ (see e.g. \cite{kristensen99}), i.e. for a.e. $t\in(0,T)$,
\begin{align*}
\|\nabla z^n(t,\cdot)\|_{L^p(Q)} & = \|\mathcal P^{\rm curl}(V_n(t,\cdot))\|_{L^p(Q)} \\
& \leq C\|\nabla y^n(t,\cdot)\|_{L^p(Q)} \\
& \leq C \sup_t \|\nabla y^n(t,\cdot)\|_{L^p(Q)}.
\end{align*}
This shows that $(\nabla z^n)$ is bounded in $L^\infty(0,T;L^p(Q))$.
To see that it generates $(\nu_{t_0,x})_x$, note that
\begin{align*}
\nabla y^n(t,\cdot) - \nabla z^n(t,\cdot) & = \nabla y^n(t,\cdot) - V_n(t,\cdot) + g^n(t,\cdot)\\
&= \nabla y^n(t,\cdot) - V_n(t,\cdot) + \mathcal P^{\rm div}(V_n(t,\cdot)) \\
& = \nabla y^n(t,\cdot) - V_n(t,\cdot) + \mathcal P^{\rm div}(V_n(t,\cdot) - \nabla y^n(t,\cdot)).
\end{align*}
However, $\mathcal P^{\rm div}$ is a weak $(1,1)$ operator (see e.g. \cite{kristensen99}), i.e. for any fixed $\eps$ and almost all $t$
\begin{equation*}
\mathcal L^d(\{|\nabla y^n(t,\cdot) - \nabla z^n(t,\cdot)|>\eps\}) \leq \frac{C}{\eps}\|\nabla y^n(t,\cdot) - V_n(t,\cdot)\|_{L^1(Q)}.
\end{equation*}
Then, by \eqref{eq:step2_3a}, it also holds that
\begin{align}
\mathcal L^{d+1}(\{|\nabla y^n - \nabla z^n| >\eps\}) & = \int_0^T \mathcal L^d(\{|\nabla y^n(t,\cdot) - \nabla z^n(t,\cdot)|>\eps\}) \nonumber \\
\label{eq:step2_4}
& \leq \frac{C}{\eps}\|\nabla y^n - V_n\|_{L^1(Q_T)} \rightarrow 0.
\end{align}
This proves that $(\nabla z^n)$ generates the measure $(\nu_{t_0,x})_x$. To prove that $(|\nabla z^n|^p)$ is equiintegrable in $Q_T$, fix $\eps>0$ arbitrary and some $q>p$. Since $(|V_n|^p)$ is equiintegrable, there exists some $(W_n)$ with the property that $\|V_n - W_n\|_{L^p(Q_T)}\leq \eps$ and $\sup_n \|W_n\|_{L^q(Q_T)} <\infty$. This is an equivalent characterisation of equiintegrability which follows immediately from the definition, see also \cite[Lemma 3.2]{kristensen99}. Then, the fact that $\mathcal P^{\rm curl}$ is a strong $(r,r)$ operator for all $1<r<\infty$, implies that
\[
\|\nabla z^n - \mathcal P^{\rm curl}(W_n)\|_{L^p(Q_T)} = \|\mathcal P^{\rm curl}(\nabla z^n -  W_n)\|_{L^p(Q_T)} \leq c \|V_n - W_n\|_{L^p(Q_T)}\leq c\eps
\]
and
\[
\sup_n\|\mathcal P^{\rm curl}(W_n)\|_{L^q(Q_T)}\leq c \sup_n \|W_n\|_{L^q(Q_T)} <\infty.
\]
This proves that $(|\nabla z^n|^p)$ is weakly relatively compact in $L^1(Q_T)$. To conclude the proof, we need to establish that $z^n$ converges strongly to $y(t_0,\cdot)$ in $L^p(Q_T)$. This is possible by exploiting the fact that $\nabla y^n$ and $V_n$ share the same oscillations and do not concentrate in $L^{q}$ with $q<p$. Indeed, we note that by \eqref{eq:step2_3a} and the bound in $L^{\infty}(0,T;L^{p}(Q))$ of $V_n$ and $\nabla y^n$ it follows that 
\begin{equation}\label{eq:s21}
\|V_n-\nabla y^n\|_{L^{r}(L^{q})}\rightarrow 0\textrm{ as }n\to\infty\textrm{ for any }r<\infty, q<p. 
\end{equation}
By adding $\nabla y^n$ to both members of \eqref{eq:hel} and taking the divergence one gets that 
\begin{equation}\label{eq:ell}
-\Delta(z^n-y^n)=\dv(\nabla y^n-V_n).
\end{equation}
Note $(z^n-y^n)$ has zero average because both $z^n$ and $y^n$ have zero average. Then, by standard elliptic estimates, we have that for any $1<q<\infty$, 
\begin{equation}\label{eq:ell1}
\|\nabla(z^n(t,\cdot)-y^n(t,\cdot))\|_{L^q(Q)}\leq C\|\nabla y^n(t,\cdot)-V_n(t,\cdot)\|_{L^q(Q)}.
\end{equation}
Let us first treat the case $d=2$ and $p=2$. In this case, let us consider $q\in (1,2)$. By Sobolev embedding and \eqref{eq:ell1} we infer that 
\begin{equation*}
\begin{aligned}
\|z^n(t,\cdot)-y^n(t,\cdot)\|_{L^2(Q)}&\leq C\|\nabla(z^n(t,\cdot)-y^n(t,\cdot))\|_{L^1(Q)}\\
                                               &\leq C\|\nabla(z^n(t,\cdot)-y^n(t,\cdot))\|_{L^q(Q)}\\
                                               &\leq C\|\nabla y^n(t,\cdot)-V_n(t,\cdot)\|_{L^q(Q)}.
\end{aligned}
\end{equation*}
Then, by integrating in time we have 
\begin{equation}\label{eq:s22}
\int_{0}^{T}\|z^n(t,\cdot)-y^n(t,\cdot)\|^{2}_{L^2(Q)}\,dt\leq C\int_{0}^{T}\| y^n(t,\cdot)-V_n(t,\cdot)\|^{2}_{L^q(Q)}\,dt\rightarrow 0,
\end{equation}
where the right-hand side goes to $0$ as $n\to\infty$ because of \eqref{eq:s21}. Since $y^{n}\rightarrow y(t_0,\cdot)$ in $L^{2}(Q_T)$, from \eqref{eq:s22}, we also have that
$z^{n}\rightarrow y(t_0,\cdot)$ in $L^{2}(Q_T)$.

On the other hand, for $d=3$, consider $q=3p/(p+3)$. Note that $q\in (1,p)$. Then, by the Sobolev embedding 
and \eqref{eq:ell1} we have that
\begin{equation*}
\begin{aligned}
\|z^n(t,\cdot)-y^n(t,\cdot)\|_{L^p(Q)}&\leq C\|\nabla(z^n(t,\cdot)-y^n(t,\cdot))\|_{L^q(Q)}\\
                                               &\leq C\| \nabla y^n(t,\cdot)-V_n(t,\cdot)\|_{L^q(Q)}
\end{aligned}
\end{equation*}
and, by integrating in time, 
 \begin{equation}\label{eq:s23}
\int_{0}^{T}\|z^n(t,\cdot)-y^n(t,\cdot)\|^{p}_{L^p(Q)}\,dt\leq C\int_{0}^{T}\| y^n(t,\cdot)-V_n(t,\cdot)\|^{p}_{L^q(Q)}\,dt\rightarrow 0.
\end{equation}                                              
Again the right-hand side goes to $0$ as $n\to\infty$ because of \eqref{eq:s21} and since $y^{n}\rightarrow y(t_0,\cdot)$ in $L^{p}(Q_T)$, from \eqref{eq:s23}, we also have that
$z^{n}\rightarrow y(t_0,\cdot)$ in $L^{p}(Q_T)$.
\end{proof}

Lastly, in the proof Theorem \ref{theorem:2}, we require another decomposition lemma which we state here. Its proof can be found in \cite{JCCVMO}.

\begin{proposition}\label{prop:3}
Let $\psi_k\rightharpoonup \psi$ in $H^1_0(Q)$. Suppose that $(\eta_k)\subset (0,1]$ and $(\eta_k\psi_k)$ is bounded in $W^{1,p}(Q)$. Then, there exist $g_k \in C^\infty_c(Q)$ and $b_k\in H^1(Q)$ such that
\begin{enumerate}
\item[(a)] $\psi_k = \psi + g_k + b_k$;
\item[(b)] $g_k,\,b_k\rightharpoonup 0$ in $W^{1,2}(Q)$ and $\eta_k g_k,\, \eta_k b_k \rightharpoonup 0$ in $W^{1,p}(Q)$;
\item[(c)] $\nabla b_k \rightarrow 0$ in measure;
\item[(d)] $\left(|\nabla g_k|^2\right)$ and $\left(|\eta_k \nabla g_k|^p\right)$ are equiintegrable.
\end{enumerate}
\end{proposition}

We may now proceed to the proof of Theorem \ref{theorem:2}.

\begin{proof}[Proof of Theorem \ref{theorem:2}] \quad
\vspace{0.2cm}

We show that there exist constants $C_0=C_0(W,K)$ and $C_1=(W,K)$ such that  for all $H\in\mathcal{F}_K$ and all $\varphi\in W^{1,p}(Q)\cap H^1_0(Q)$, 
\begin{equation}\label{eq:step1}
\int_Q |V(\nphi)|^2 dx
\leq C_1 \int_Q \GF(H(x),\nphi) dx +  C_0 \int_Q |V(\varphi)|^2 dx.
\end{equation}
Then, choose $H = \bar{F}(t_0,\cdot) = \nabla\bar{y}(t_0,\cdot)\in\mathcal{F}_K$ for some $K>0$ uniform in $t_0$, and $\varphi = z^k(t,\cdot) - \bar{y}(t_0,\cdot)$ where $z_k$
is constructed in Lemma \ref{lemma:decomposition}. Next, integrate in time and take the limit $k\to\infty$, using the equiintegrability of $(|\nabla z^k|^p)$ and that $(\nabla z^k)$ generates $(\nu^F_{t_0,x})_{x\in Q}$, to conclude the proof of Theorem 5.1, i.e. that
\begin{multline*}
\int_Q\langle\nu^F_{t_0,x},|V(\xi-\bar F(t_0,x))|^2\rangle dx \\
\leq C_1 \int_Q  \langle\nu_{t_0,x}, G_W(\bar{F}(t_0,x),\xi - \bar{F}(t_0,x))\rangle \, dx +  C_0 \int_Q |V(y(t_0,x)-\bar y(t_0,x))|^2dx.
\end{multline*}

In order to show \eqref{eq:step1}  it suffices to prove the existence of some $\eps_0>0$ such that for all $H\in\mathcal{F}_K$ and all $\varphi\in W^{1,p}(Q)\cap H^1_0(Q)$ with $\|\varphi\|_{L^p(Q)}<\eps_0$ it holds that
\begin{equation}\label{eq:tGF>0}
\int_Q\tGF(H(x),\nphi) + \frac{c_1}{2}|\varphi|^2 \geq 0.
\end{equation}
Indeed, by the definition of $\tW$, the strong convexity of $f(\xi) = |V(\xi)|^2$ and Lemma \ref{lemma:technical_general} (d), \eqref{eq:tGF>0} says that whenever $\|\varphi\|_{L^p(Q)}<\eps_0$,
\[
C(K)\int_Q |V(\nphi)|^2 \leq C(K) \int_Q G_f(H,\nphi) \leq \int_Q\GF(H,\nphi) + \frac{c_1}{2}|V(\varphi)|^2.
\]
Then, we can conclude \eqref{eq:step1} as for $\|\varphi\|_{L^p}\geq \eps_0$, by the coercivity of $W$ and Young's inequality, it holds that
\begin{align}
\int_Q\GF(H(x),\nabla\varphi) &\geq \int_Q c|H+\nabla\varphi|^p - C(W,K) - C(\delta)|DW(H)|^q - \delta |\nabla\varphi|^p\nonumber\\
& \geq  - C(W,K) + \tilde{c} \int_Q |\nabla\varphi|^p \geq  - \frac{C(W,K)}{\eps^p_0}\int_Q|\varphi|^p + \tilde{c} \int_Q |\nabla\varphi|^p,
\label{eq:1.1}
\end{align}
for $\delta$ small enough. This concludes the proof after noting that $\|V(\nabla\varphi)\|^2_{L^2} \leq 1 + 2\|\nabla\varphi\|^p_{L^p}$ and that, by Poincar\'e's inequality, $\eps^p_0 \leq  C \|\nabla\varphi\|^p_{L^p}$.

Hence, we are left to prove \eqref{eq:tGF>0} where we proceed by contradiction. Suppose \eqref{eq:tGF>0} is false. Then we can find $(H_k)\subset \mathcal{F}_K$, $H\in\mathcal{F}_K$, and $(\varphi_k)\subset W^{1,p}(Q)\cap H^1_0(Q)$ such that $\|\varphi_k\|_{L^p(Q)} \rightarrow 0$, $H_k \rightarrow H$ in $C^0({Q})$ 
and
\begin{equation}\label{eq:6}
\int_Q \tGF(H_k(x),\nphi_k(x)) + \frac{c_1}2|\varphi_k(x)|^2 <0.
\end{equation}

\underline{Step 1}: We show that $\varphi_k \rightarrow 0$ in $W^{1,p}(Q)$ and that
\begin{equation}\label{eq:Lambdabounded}
\sup_k \frac{\beta_k^p}{\alpha_k^2} =:\Lambda <\infty,\mbox{ where }\alpha_k = \|\nphi_k\|_{L^2(Q)},\,\,\beta_k = \|\nphi_k\|_{L^p(Q)}.
\end{equation}

By \eqref{eq:6}, after using the $p$-coercivity of $\tW$ and Young's inequality, we find that 
$(\nabla\varphi_k)$ is bounded in $W^{1,p}(Q)$.
We may thus apply Proposition \ref{prop:2} with $a_k=1$ and $h_k = \varphi_k$ to find that, by \eqref{eq:6},
\begin{align*}
\liminf_{k\to\infty} \frac{c_0}4\int_Q|V(\nphi_k)|^2 & \leq \liminf_{k\to\infty}\int_Q\tGF(H_k,\nphi_k) \leq 0,
\end{align*}
and $\varphi_k \rightarrow 0$ in $W^{1,p}(Q)$. 
Regarding \eqref{eq:Lambdabounded}, Lemma \ref{lemma:technical_general} (c) and the $p$-coercivity of $\tW$ implies that 
\begin{equation}\label{eq:new_coercivity}
\int_Q\tGF(H(x),\nphi_k) \geq d\int_Q |\nphi_k|^p - c\int_Q |\nphi_k|^2,
\end{equation}
where $d=d(\tW,K)$, $c=c(\tW,K)$. Then \eqref{eq:Lambdabounded} follows after dividing by $\alpha_k^2$ and noting \eqref{eq:6}.

\underline{Step 2}: \textcolor{black}{Since $\varphi_k\to 0$ in $W^{1,p}(Q)$ we cannot contradict \eqref{eq:6}. Instead, let
\[
\psi_k := \alpha_k^{-1} \varphi_k
\]
and decompose into purely oscillating and concentrating parts using Proposition \ref{prop:3}. Indeed, note that since} $\|\nabla\psi_k\|_{L^2(Q)} = 1$ and $\psi_k\in H^1_0(Q)$, we find that 
$\psi_k\rightharpoonup \psi$ in $W^{1,2}(Q)$. Moreover, 
setting $\eta_k = \frac{\alpha_k}{\beta_k}\in(0,1]$,
we have that $(\eta_k\psi_k)$ is bounded in $W^{1,p}(Q)$.
We may thus decompose $\psi_k$ to find $g_k \in C^\infty_c(Q)$, $b_k\in H^1(Q)$ as in Proposition \ref{prop:3}.
Write
\begin{equation}\label{eq:7}
f_k(x) = \alpha_k^{-2}\left[\tGF(H_k,\alpha_k\nabla\psi_k) - \tGF(H_k,\alpha_k\nabla b_k) \right]
\end{equation}
and note that, since $\alpha_k\psi_k = \varphi_k$, by \eqref{eq:6},
\begin{equation}\label{eq:7a}
\int_Q f_k(x) + \alpha_k^{-2} \tGF(H_k,\alpha_k\nabla b_k) + \frac{c_1}2|\psi_k|^2<0.
\end{equation}
\textcolor{black}{Following the idea of the proofs in \cite{JCCVMO,JCCKK,GM} we show that, in the limit, the contribution of the purely concentrating part $\alpha_k^{-2} \tGF(H_k,\alpha_k\nabla b_k)$ in \eqref{eq:7a} is non-negative due to quasiconvexity. This follows by applying Proposition \ref{prop:2} with $a_k = \alpha_k$ and $h_k = \alpha_k b_k$ to the term $\alpha_k^{-2} \tGF(H_k,\alpha_k\nabla b_k)$, after noting that
\begin{equation}\label{eq:8Lambda}
\alpha_k^{p-2} = \frac{\beta_k^p}{\alpha_k^2}\eta_k^p = \Lambda \eta_k^p,
\end{equation}
where, by Step 1, $\Lambda = \beta_k^p/\alpha_k^2$ is bounded abd thus
\[
a_k^{-2}|V(\alpha_kb_k)|^2 = |b_k|^2 + \Lambda |\eta_k b_k|^p \rightarrow 0\mbox{ in }L^1(Q).
\]
Also, $a_k^{-2}|V(\nabla h_k)|^2 = |\nabla b_k|^2 + \Lambda |\eta_k \nabla b_k|^p$ which is bounded in $L^1(Q)$. So, Proposition \ref{prop:2} says that
\begin{align*}
0\leq \liminf_{k\to\infty} \frac{c_0}4\int_Q\alpha_k^{-2} |V(\alpha_k \nabla b_k)|^2 \leq \liminf_{k\to\infty} \alpha_k^{-2}\int_Q\tGF(H_k,\alpha_k \nabla b_k).
\end{align*}
In particular,
\begin{equation}\label{eq:8}
\frac{c_1}2\int_Q |\psi|^2 + \liminf_{k\to\infty} \int_Q f_k(x) \leq 0.
\end{equation}
}
\underline{Step 3}: Let $\nu = (\nu_x)_{x\in Q}$ be the $W^{1,2}$ gradient Young measure generated by the sequence $\psi_k$ and recall that $H_k \rightarrow H$ in $C^0(Q)$. We show that
\begin{equation}\label{eq:8a}
\frac12 \int_Q \langle \nu_x, D^2\tW(H(x))\xi:\xi\rangle \leq \liminf_{k\to\infty} \int_Q f_k(x).
\end{equation}
In particular, in conjunction with \eqref{eq:8}, we infer that
\begin{equation}\label{eq:9}
\frac{1}2\int_Q c_1 |\psi|^2 +  \langle \nu_x, D^2\tW(H(x))\xi:\xi\rangle \leq 0.
\end{equation}
\textcolor{black}{In Step 4 we show that \eqref{eq:9}, in conjunction with Proposition \ref{prop:1}, leads to a contradiction.}

To prove \eqref{eq:8a} we show the equiintegrability of $(f_k)$ defined in \eqref{eq:7}. Indeed,
by Lemma \ref{lemma:technical_general} (a) and for a constant $C=C(\tW,K)$, Young's inequality gives
\begin{align*}
|f_k| & \leq C(|\nabla\psi_k|+ |\nabla b_k| + \alpha_k^{p-2}|\nabla\psi_k|^{p-1} + \alpha_k^{p-2}|\nabla b_k|^{p-1})|\nabla\psi_k - \nabla b_k|\\
& \leq \delta C (|\nabla\psi_k|^2+ |\nabla b_k|^2) + C(\delta)| \nabla(\psi+g_k)|^2\\
&\quad + \delta C(\alpha_k^{p-2}|\nabla\psi_k|^p+ \alpha_k^{p-2}|\nabla b_k|^p) + C(\delta) \alpha_k^{p-2} | \nabla(\psi+g_k)|^p,
\end{align*}
recalling that, by Proposition \ref{prop:3}, $\nabla\psi_k - \nabla b_k = \nabla(\psi+g_k)$. However, by Proposition \ref{prop:3}, $\psi_k$ and $b_k$ are bounded in $W^{1,2}(Q)$, whereas $(|\nabla(\psi+g_k)|^2)$ is equiintegrable. Similarly, since $\alpha_k^{p-2} = \Lambda \eta_k^p$ we infer that $\alpha_k^{p-2}|\nabla\psi_k|^p$ and $\alpha_k^{p-2}|\nabla b_k|^p$ are bounded, whereas $\alpha_k^{p-2} | \nabla(\psi+g_k)|^p$ is equiintegrable. Hence, $(f_k)$ is also equiintegrable and for $\eps > 0$ fixed, we can find $m_\eps$ such that 
\begin{equation}\label{eq:10}
\int_Q f_k > -\eps + \int_{\{|\nabla\psi_k| < m\}\cap\{|\nabla b_k|< m\}} f_k,\quad\forall\, m\geq m_\eps.
\end{equation}
This follows as $\nabla b_k\rightarrow 0$ in measure and $\lim_{r\to\infty} \sup_k \{|\nabla\psi_k|>r\} = 0$. Also, 
since $\int_Q \langle\nu_x, |\xi|^2\rangle < \infty$, we may assume that for all $m\geq m_\eps$,
\begin{equation}\label{eq:12}
\int_Q \langle \nu_x, D^2\tW(H)\xi:\xi \rangle = \int_Q \langle \nu_x, D^2\tW(H)\xi:\xi\chi_{B(0,m)}(\xi) \rangle + \eps,
\end{equation}
where $\chi_{A}$ denotes the indicator function of a set $A\subset\Rdd$. Since $B(0,m)$ is open, for all $x\in Q$ the function
$
\xi\mapsto D^2\tW(H(x))\xi:\xi \chi_{B(0,m)(\xi)}
$
is lower semicontinuous and, as $(\nabla\psi_k)$ generates $(\nu_x)_{x\in Q}$ and $H_k \rightarrow H$ in $C^0(Q)$, we deduce that
\begin{multline}\label{eq:13}
\int_Q \langle \nu_x, D^2\tW(H)\xi:\xi \chi_{B(0,m)}(\xi) \rangle \leq \liminf_{k\to\infty}\int_{\{|\nabla\psi_k|<m\}}D^2\tW(H)\nabla\psi_k:\nabla\psi_k\\
= \liminf_{k\to\infty}\int_{\{|\nabla\psi_k|<m\}}D^2\tW(H_k)\nabla\psi_k:\nabla\psi_k.
\end{multline}
Combining \eqref{eq:13} with \eqref{eq:12}, we now infer that for all $m\geq m_\eps$
\begin{equation}\label{eq:14}
\int_Q \langle \nu_x, D^2\tW(H)\xi:\xi \rangle \leq \liminf_{k\to\infty}\int_{\{|\nabla\psi_k|<m\}}D^2\tW(H_k)\nabla\psi_k:\nabla\psi_k+\eps.
 \end{equation}
To conclude the proof, we next claim that
\begin{equation}\label{eq:15}
\frac12 \liminf_{k\to\infty}\int_{\{|\nabla\psi_k|<m\}}D^2\tW(H_k)\nabla\psi_k:\nabla\psi_k
= \lim_{k\to\infty}\int_{\{|\nabla\psi_k|<m\}\cap \{|\nabla b_k|<m\}} f_k.
\end{equation}
Before proving \eqref{eq:15}, note that in conjunction with \eqref{eq:14} and \eqref{eq:10} it says that
\[
\frac12 \int_Q \langle \nu_x, D^2\tW(H(x))\xi:\xi \rangle \leq \liminf_{k\to\infty}\int_Q f_k + \frac{3\eps}{2}.
\]
By taking  $\eps\to 0$, \eqref{eq:8a} follows. To prove \eqref{eq:15},
set
$
A_k :=\{|\nabla\psi_k|<m\}$
and $B_k:=\{|\nabla b_k|<m\}$,
so that
\begin{align*}
&\chi_{A_k\cap B_k} f_k  = \chi_{A_k\cap B_k}\int_0^1(1-s) \left[D^2\tW(H_k+ s\alpha_k\nabla\psi_k) - D^2\tW(H_k)\right]\nabla\psi_k:\nabla\psi_k\,ds\\
&+ \chi_{A_k} \frac12 D^2\tW(H_k)\nabla\psi_k:\nabla\psi_k - \chi_{A_k} \frac12 D^2\tW(H_k)\nabla\psi_k:\nabla\psi_k \left(1 - \chi_{B_k}\right)\\
& - \chi_{A_k\cap B_k} \int_0^1(1-s) D^2\tW(H_k+ s\alpha_k\nabla b_k)\nabla b_k:\nabla b_k \,ds =: I_1^k + I_2^k + I_3^k + I_4^k.
\end{align*}
Hence, it suffices to show that $I_i^k\to 0$, for $i=1,3,4$, as $k\to\infty$
which follows by dominated convergence as $\alpha_k\to 0$, $H_k \to H$ in $C^0(Q)$ and $\nabla b_k \to 0$ in measure.

\underline{Step 4}: We employ Proposition \ref{prop:1}, combined with \eqref{eq:9}, to reach a contradiction. By \eqref{eq:qc_rc}, the function $\xi\mapsto D^2\tW(H(x))\xi:\xi$
is quasiconvex for each $x\in Q$. Since 
 $(\nu_x)_{x\in Q}$ is a gradient Young measure, Jensen's inequality implies
\[
\int_Q c_1|\psi|^2 + D^2\tW(H(x))\nabla\psi:\nabla\psi \leq \int_Q c_1|\psi|^2 + \langle\nu_x, D^2\tW(H(x))\xi:\xi\rangle \leq 0,
\]
by \eqref{eq:9}, after adding $c_1|\psi|^2$ and integrating over $Q$. However, by Proposition \ref{prop:1},
\[
\int_Q c_1|\psi|^2 + D^2\tW(\bar{F}(x))\nabla\psi:\nabla\psi \geq \frac{c_0}2\int_Q|\nabla\psi|^2,\quad\forall\psi\in W^{1,2}(Q),
\]
i.e. $\nabla\psi = 0$ and, since $\psi\in H^1_0(Q)$, $\psi = 0$. Thus, recalling Step 1, we may argue as in Step 2 and apply Proposition \ref{prop:2} with $a_k = \alpha_k$ and $h_k = \alpha_k\psi_k$, to infer that
\begin{align*}
0 < \frac{c_0}4 
& \leq \liminf_{k\to\infty} \frac{c_0}4 \int_Q |\nabla\psi_k|^2 + \alpha_k^{p-2}|\nabla\psi_k|^p
 \leq \liminf_{k\to\infty}\alpha_k^{-2}\int_Q \tGF(H_k,\alpha_k \nabla\psi_k)\\
& = \liminf_{k\to\infty}\alpha_k^{-2}\int_Q \tGF(H_k,\nphi_k) + \frac{c_1}{2} |\varphi_k|^2 \leq 0,
\end{align*}
by \eqref{eq:6} as $\alpha_k^{-1}\varphi_k = \psi_k\to 0$. This contradiction concludes the proof.
\end{proof}

        





 \section*{Acknowledgement}
The authors would like to thank John M. Ball and Jan Kristensen for useful discussions on the present work.
S. Spirito acknowledges the support by INdAM-GNAMPA Project: {\em Analisi di Modelli Matematici della Fisica, della Biologia e delle Science Sociali}. This work was partially written when both authors were affiliated to the GSSI - Gran Sasso Science Institute, L'Aquila, Italy and they would like to acknowledge the Institute's support.



\frenchspacing

\end{document}